\documentclass[11pt]{article}
\usepackage{amsmath,amsfonts,ntheorem}
\usepackage{txfonts,wasysym,lscape,amssymb,mathrsfs,extarrows}
\usepackage[all,pdf]{xy}
\usepackage{graphicx}
\usepackage{subfigure}
 \usepackage{multirow}
 \usepackage{rotating}
 \usepackage[table,xcdraw]{xcolor}
 \usepackage{indentfirst}
 \usepackage{algorithm}
 \usepackage{algorithmic}
 \usepackage{enumerate}
\RequirePackage{color,graphicx}

\newtheorem{theorem}{Theorem}[section]
\newtheorem{proposition}{Proposition}[section]
\newtheorem{lemma}{Lemma}[section]
\newtheorem{corollary}{Corollary}[section]
\newtheorem{definition}{Definition}[section]

\newtheorem{remark}{Remark}[section]

\title{Stability of Lagrangian Generalized Nash Equilibriums}

\author{
Lixin Tang\footnote{National Frontiers Science Center for Industrial Intelligence and Systems Optimization, Northeastern University, Shenyang 110819, P. R. China.} \footnote{Key Laboratory of Data Analytics and Optimization for Smart Industry (Northeastern University), Ministry of Education, Shenyang 110819, P. R. China. {\sl Email}: qhjytlx@mail.neu.edu.cn. This author was supported by the Major Program of National Natural Science Foundation of China (Nos. 72192830 and 72192831) and the 111 Project (B16009).} \quad and \quad
Liwei Zhang \footnote{National Frontiers Science Center for Industrial Intelligence and Systems Optimization, Northeastern University, Shenyang 110819, P. R. China.} \footnote{Key Laboratory of Data Analytics and Optimization for Smart Industry (Northeastern University), Ministry of Education, Shenyang 110819, P. R. China. {\sl Email}: zhanglw@mail.neu.edu.cn.
 This author was supported National Key R\&D Program of China under project No. 2022YFA1004000,  the Major Program of National Natural Science Foundation of China (Nos. 72192830 and 72192831), National Natural Science Foundation of China (No.12371298) and the 111 Project (B16009).}}

\date{}
\begin{document}

\maketitle

\begin{abstract}
Lagrangian generalized Nash equilibriums (LGNEs) were introduced by Rockafellar (2024) for a class of generalized Nash equilibrium problems (GNEPs) in which each player's strategy is subject to conic constraints. This paper investigates the stability properties of the LGNE solution set, specifically focusing on the Aubin property, isolated calmness, and Lipschitz continuous single-valued localization.
For general conically constrained GNEPs, characterizations of the Aubin property and isolated calmness of the LGNE solution mapping under canonical perturbations are established. These characterizations are formulated using the coderivative and graph derivative of normal cone mappings.
Subsequently, these general results are specialized to GNEPs with equality and inequality constraints, yielding explicit characterizations for both the Lipschitz continuous single-valued localization and isolated calmness of the corresponding LGNE solution mapping, which are described by nonsingularity of linear complementarity sytems.
For GNEPs with shared conic constraints, the Aubin property and isolated calmness of the consensus LGNE solution mapping--where identical Lagrange multipliers are assigned to the shared constraint--are first characterized. We further analyze the case when the conic constraints are specialized as equalities and inequalities.
Finally, for classical conically constrained Nash equilibrium problems, the Aubin property and isolated calmness of the Lagrangian Nash equilibrium solution mapping are also analyzed.

\vskip 6 true pt \noindent \textbf{Key words}:
Generalized Nash equilibrium problems, Lagrangian generalized Nash equilibriums,consensus Lagrangian generalized Nash equilibriums, Aubin property, isolated calmness, Lipschitz continuous single-valued localization.

\vskip 12 true pt \noindent \textbf{AMS subject classification}:
90C30, 90C31, 49J53
\end{abstract}
\bigskip\noindent

\section{Introduction}\label{Sec1}
 \setcounter{equation}{0}
Different from standard  Nash equilibrium (see Nash\cite{Nash51}) problems, generalized Nash equilibrium problems (GNEPs) allow a dependence of the strategy set of one player from
the decisions of the other players, see Arrow and  Debreu \cite{Arrow54}.
Formally, the GNEP consists of $N$ players, each player $\nu$  has the variables
$x^{\nu}\in \mathbb R^{n_{\nu}}$, whose strategy must belong to a set
 that
depends on the rival players' strategies.  We denote by $x$  the vector generated by all these decision variables:
$
x=\left(
x^{1},\cdots,x^{N}
\right),
$
which obviously has dimension $n=\sum_{\nu=1}^N n_{\nu}$, and by $x^{-\nu}$
 the vector formed by all the players'
decision variables except those of player $\nu$. To emphasize the $\nu$th player's variables
within $x$, we sometimes write $(x^{\nu},x^{-\nu})$ instead of $x$.
Each player has an objective function $\theta_{\nu}:\mathbb R^n \rightarrow \mathbb R$
 that depends on both his own
variables $x^{\nu}$ as well as on the variables $x^{-\nu}$
 of all other players. This mapping $\theta_{\nu}$ is
often called the utility function of player $\nu$, sometimes also the payoff function or loss
function, depending on the particular application in which the GNEP arises.
Furthermore, each player's  strategy has the feasible set $C_{\nu}(x^{-\nu})\subset \mathbb R^{n_{\nu}}$
 that depends on  the variables $x^{-\nu}$
 of all other players. The aim of player $\nu$, given the other players' strategies $x^{-\nu}$,
 is to
choose a strategy $x^{\nu}$ that solves the minimization problem
\begin{equation}\label{eq:p1}
\begin{array}{ll}
\min_{x_{\nu}}& \theta_{\nu}(x^{\nu},x^{-\nu})\\[4pt]
{\rm s.t.} & x_{\nu}\in C_{\nu}(x^{-\nu}).
\end{array}
\end{equation}
For any $x^{-\nu}$, the solution set of problem (\ref{eq:p1}) is denoted by $S_{\nu}(x^{-\nu})$.
The generalized Nash equilibrium (GNE) is
 a vector $\bar x\in \mathbb R^n$ such that
$$
\bar x^{\nu}\in S_{\nu}(\bar x^{-\nu}) \mbox{  for all }\nu = 1,\ldots, N.
$$
Such a point $\bar x$ is called a (generalized Nash) equilibrium or, more simply, a solution of
the GNEP.  In essence, Generalized Nash Equilibrium is the natural framework for analyzing industrial systems where agents are not just competing in outcomes (price, market share) but are also fundamentally coupled in their feasible actions by physical, regulatory, or contractual shared limits. Its application to quality stability demonstrates how it can move analysis from simple bilateral contracts to modeling the systemic stability of a multi-firm ecosystem under common performance goals. This makes GNEP a critical tool for optimizing complex, modern industrial networks.  A survey of GNEPs by  Facchinei and Kanzow (2010)\cite{FKanzow2010}  explained clearly  the background,  Kurash-Kuhn-Tucker conditions, stability of solutions and various numerical algorithms for GNEPs.

In this paper, we consider the GNEP when $C_{\nu}(x^{-\nu})$ is defined by a conic constraint, specified as follows
\begin{equation}\label{eq:coneC}
C_{\nu}(x^{-\nu})=\{x^{\nu}\in X_{\nu}:g_{\nu}(x^{\nu},x^{-\nu})\in K_{\nu}\},
\end{equation}
where $X_{\nu}\subset \mathbb R^{n_{\nu}}$ is a nonempty closed convex set, $K_{\nu}\subset {\cal Y}_{\nu}$ is closed convex cone and ${\cal Y}_{\nu}$ is a finitely dimensional Hilbert space. The Lagrangian function associated with (\ref{eq:p1}) with $C_{\nu}(x^{-\nu})$  in
(\ref{eq:coneC})
is defined by
$$
    l_{\nu}  (x^{\nu}, x^{-\nu}, y^{\nu}): = \theta_{\nu} (x^{\nu},x^{-\nu})
    +\langle y^{\nu}, g_{\nu}(x^{\nu},x^{-{\nu}})\rangle.
$$

To address the stability properties, we formally introduce a general parameter vector $w \in \mathbb R^d$ and consider the following generalized Nash equilibrium problem with canonically perturbations:
\begin{equation}\label{GNEP-p}
    \textbf{P}_{\nu} (x^{-{\nu}},w,v^{\nu},u^{\nu})\quad \quad\quad \quad
    \begin{array}{ll}
     \min\limits_{x^{\nu} \in X_{\nu}} & f_{\nu}  (x^{\nu}, x^{-{\nu}}; w)-\langle v^{\nu} , x^{\nu}\rangle\\[4pt]
   {\rm s.t.}& F_{\nu} (x^{\nu}, x^{-{\nu}}; w) - u^{\nu}\in K_{\nu},
   \end{array}
   \end{equation}
where $w \in \mathbb R^d$, $v^{\nu} \in \mathbb{R}^{n_{\nu}}$ and $ u^{\nu}\in {\cal Y}_{\nu}$ are perturbation parameters. For simplicity, we denote
$p=(w;v^1;\ldots; v^{N}; u^1;\ldots;u^{N})\in \mathbb R^d \times \mathbb R^{n_1}\times \cdots \times \mathbb R^{n_N}\times {\cal Y}_1\times  \cdots \times  {\cal Y}_N$.
Let $\bar w \in \mathbb R^d$ be a fixed parameter such that
$$
f_{\nu} (x^{\nu}, x^{-{\nu}};\bar w)=\theta_{\nu}(x^{\nu}, x^{-{\nu}})\quad  \mbox{ and } \quad F_{\nu} (x^{\nu}, x^{-{\nu}};\bar w)=g_{\nu}(x^{\nu}, x^{-{\nu}}).
$$
and the Lagrangian function $l_{\nu}$ becomes
$$
    l_{\nu}  (x^{\nu}, x^{-\nu}, y^{\nu}) = f_{\nu} (x^{\nu},x^{-\nu};\bar w)
    +\langle y^{\nu}, F_{\nu}(x^{\nu},x^{-{\nu}};\bar w)\rangle.
$$
In this setting,  the $\nu$-th player optimization problem (\ref{GNEP-p}) is a perturbed problem of the problem defined by (\ref{eq:p1}) and (\ref{eq:coneC}) when
$(w,v^k,u^k)$ is perturbed from $(\bar w,0,0)$.

We denote the feasible set of the $\nu$-th player optimization problem (\ref{GNEP-p}) by $C_{\nu}(x^{-\nu},w,u^{\nu})$, namely
$$
C_{\nu}(x^{-\nu},w,u^{\nu})=\{x^{\nu}\in X_{\nu}: F_{\nu} (x^{\nu}, x^{-{\nu}}; w) - u^{\nu}\in K_{\nu}\}.
$$
 The generalized Nash equilibrium (GNE) for fixed parameter $p$ is a solution to
\begin{equation}\label{NE}
    {\rm finding} \ \tilde{x}\in \mathbb{R}^n\  {\rm such\  that\
    each\ } \tilde{x}^k\  {\rm is\ a\ minimizer\ of\ } \textbf{P}_k (\tilde{x}^{-k},w,v^k,u^k),\, k=1,\ldots, N.
    \end{equation}
The local generalized Nash equilibrium  for parameter $p$ is defined as a solution to
\begin{equation}\label{LGNE}
    {\rm finding} \ \tilde{x}\in \mathbb{R}^n\  {\rm such\  that\
    each\ } \tilde{x}^k\  {\rm is\ a\ local\  minimizer\ of\ } \textbf{P}_k (\tilde{x}^{-k},w,v^k,u^k),\, k=1,\ldots, N.
    \end{equation}
The Lagrangian function associated with (\ref{GNEP-p}) is defined by
$$
    L_{\nu}  (x^{\nu}, x^{-\nu}, y^{\nu} ;w,v^{\nu},u^{\nu}): = f_{\nu} (x^{\nu},x^{-\nu}; w)  -\langle v^{\nu} , x^{\nu}  \rangle
    +\langle y^{\nu}, F^{\nu}(x^{\nu},x^{-{\nu}}; w)-u^{\nu}\rangle
$$
for $\nu=1,...,N$. Under the basic constraint qualification for problem $\textbf{P}_{\nu} (\bar x^{-{\nu}},\bar w,0,0)$  that
$$
\left.
\begin{array}{r}
0\in N_{X_{\nu}}(\bar x^{\nu})+{\rm D}_{x^{\nu}}F_{\nu}(\bar x,\bar w)^*y^k=0\\[4pt]
y^k \in N_{K_{\nu}}(F_{\nu}(\bar x,\bar w))
\end{array}
\right
\} \Rightarrow y^k=0,
$$
there is the formula (c.f.  \cite[Theorem 6.14]{RW98}) that
$$
N_{C_{\nu}(\bar x^{-\nu}, \bar w,0)}(\bar x^{\nu})=N_{X_{\nu}}(\bar x^{\nu})+{\rm D}_{x^{\nu}}F_{\nu}(\bar x,\bar w)^*N_{K_{\nu}}(F_{\nu}(\bar x,\bar w)).
$$
The necessary optimality condition for $\bar x^{\nu}$ being a local solution of $\textbf{P}_{\nu} (\bar x^{-{\nu}},\bar w,0,0)$ is that
\begin{equation}\label{KKT-original}
0\in \nabla_{x^{\nu}}f_{\nu}(\bar x, \bar w) + N_{C_{\nu}(\bar x^{-\nu}, \bar w,0)}(\bar x^{\nu}).
\end{equation}
Noting that each $K_{\nu}$ is a nonempty closed convex cone, we obtain that condition (\ref{KKT-original}) is equivalent to
\begin{equation}\label{KKT-originalD}
\begin{array}{l}
0\in \nabla_{x^{\nu}}L_{\nu}(\bar x, y^{\nu};\bar w,0,0)+N_{X_{\nu}}(\bar x^{\nu}),\\[4pt]
0\in -\nabla_{y^{\nu}}L_{\nu}(\bar x, y^{\nu};\bar w,0,0)+N_{K_{\nu}^{\circ}}(y^{\nu}),
\end{array}
\end{equation}
where $K_{\nu}^{\circ}$ is the polar of $K_{\nu}$.

  This brings us to the notion of a Lagrangian generalized Nash equilibrium (LGNE) by Rockafellar (2024) \cite{Rock24}. The set of LGNEs, denoted here by  $S_{\rm Lag}: \mathbb{R}^d\times \mathbb R^{n}\times {\cal Y}\rightarrow \mathbb{R}^{n} \times {\cal Y}$,
   of the GNEP  whose $\nu$-player solves
   $\textbf{P}_{\nu} (x^{-{\nu}},w,v^{\nu},u^{\nu})$ for fixed  $p$ is defined as

\begin{equation}\label{KKT-mapping}
S_{\rm Lag}(p)=\left
\{(x,y)\in \mathbb R^n \times {\cal Y}:\left(
\begin{array}{l}
0\in \nabla_{x^{\nu}}L_{\nu}(x, y^{\nu};w,v^{\nu},u^{\nu})+N_{X_{\nu}}(x^{\nu})\\[4pt]
0\in -\nabla_{y^{\nu}}L_{\nu}(x, y^{\nu};w,v^{\nu},u^{\nu})+N_{K_{\nu}^{\circ}}(y^{\nu})
\end{array} \right),\nu=1,\ldots,N
\right\},
\end{equation}
 where ${\cal Y}= {\cal Y}_1 \times \cdots\times  {\cal Y}_N$.

The primary advantages of LGNE are  computational tractability (it can be computed), interpretability (it has a clear price-based meaning), and refinement (it selects a sensible equilibrium from a potentially large set). For most industrial and economic applications-from environmental markets to network resource allocation-when analysts talk about computing or implementing a GNE, they are almost always specifically seeking a LGNE.

 When considering constrained optimization problems, the set-valued mapping for LGNEs-defined by (\ref{KKT-mapping})-reduces to the KKT solution mapping \cite{DSZhang2017}, which has been extensively investigated over the past 45 years. Research on the stability properties of the KKT solution mapping in nonlinear programming is relatively mature. Robinson (1980) \cite{Robinson1980} established that the strong second-order sufficient condition (SSOSC) and the linear independence constraint qualification (LICQ) together imply the strong regularity of solutions to the KKT system. Notably, the converse also holds, as shown by Jongen et al. (1990) \cite{Jongen1990}.
Subsequently, Robinson (1982) \cite{Robinson1982} demonstrated that the second-order sufficient condition (SOSC) and the Mangasarian-Fromovitz constraint qualification (MFCQ) guarantee the upper Lipschitz continuity of KKT solutions. Dontchev and Rockafellar (1997) \cite{DRockafellar97} further showed that the strict MFCQ and second-order sufficient optimality conditions are equivalent to the robust isolated calmness of the KKT system.
For second-order conic optimization, key results on the stability properties of the KKT solution mapping are available in \cite{BR05, WZ09, ZZWW17}. Similarly, for nonlinear semidefinite optimization, relevant findings on this topic can be found in \cite{Sun06, ChanS08, ZZhang16, DSZhang2017}.

 Among existing studies on Nash equilibrium problems (NEPs), results concerning stability remain relatively limited. The first work focusing on the stability of the KKT solution mapping $S_{{\rm Lag}}$ was conducted by Kojima et al. (1985) \cite{Kojima85}, who established a criterion for the strong regularity of the KKT mapping in multi-person noncooperative games. Specifically, the problem they considered was a classical NEP where each player faced simple linear constraints.
After 40 years, Diao et al. (2025) \cite{DDZhang2025} investigated the stability properties of the KKT solution mapping $S_{{\rm Lag}}$
 for classical NEPs under canonical perturbations, with each player subject to equality and inequality constraints.
Separately, Rockafellar (2024) \cite{Rock24} studied three stability notions for generalized Nash equilibrium problems (GNEPs): full stability, tilt stability, and the near-tilt stability (or near-full stability) of local equilibria.
To date, no stability results regarding LGNEs for GNEPs have been reported. This gap in the literature motivates our investigation into the stability properties of LGNEs for GNEPs.

 The paper is organized as follows. In Section 2, we establish characterizations of the Aubin property and isolated calmness of LGNEs for the GNEP defined by (\ref{eq:p1}) and (\ref{eq:coneC}), and apply these characterizations to the scenario where  $K_{\nu}$
 is taken to be a polyhedral convex cone.
In Section 3, we focus on GNEPs with shared constraints and investigate the stability of the set of consensus LGNEs -- those in which all players share the same multipliers for the shared constraints. In Section 4, we examine the Aubin property and isolated calmness of LGNEs for classical Nash equilibrium problems (NEPs). Section 5 offers concluding remarks.

\section{Stability  of Lagrangian generalized Nash equilibriums}
\setcounter{equation}{0}

For developing stability properties of $S_{\rm Lag}$, we need to understand stability notions of a general set-valued mapping from \cite{RW98} or \cite{DRockafellar2009}. Here we only list four  stability concepts of a set-valued mapping $F:{\cal  X}\rightrightarrows {\cal Y}$ from a finite-dimensional Hilbert space ${\cal X}$ to another  finite-dimensional Hilbert space ${\cal Y}$, which are used in our discussions.

\begin{definition}\label{def-Lipl}
The set-valued mapping $F:{\cal  X}\rightrightarrows {\cal Y}$
 is said to have a single-valued and Lipschitz continuous localization at
$\bar x$  for $\bar y \in F(\bar x)$  if there exist open neighborhoods $U$ of $\bar x$ and $V$ of $\bar y$
 such that the localization
$s:U \rightarrow V$ of $F$ with
${\rm gph}\,s={\rm gph}\, F \cap (U\times V)$ is single-valued and Lipschitz continuous.
\end{definition}
\begin{definition}\label{Aubin property} The set-valued mapping $F:{\cal  X}\rightrightarrows {\cal Y}$ is said to have  Aubin property around $(\bar x,\bar y)\in {\rm gph}\, F$ if
    $$
    F(x)\cap V \subset F(u)+L\|u-x\|\textbf{B}_{\cal Y},\, {\rm for all }\,\, (x,u)\in U\times U,
    $$
for neighborhoods $U$ of $\bar x$ and  $ V $ of $\bar y$ and a constant $L>0$.
\end{definition}
\begin{definition}\label{calmness-s}
   The  multi-valued mapping  $F:{\cal X}\rightrightarrows {\cal Y}$ is said to be calm  at  $\bar x$  if there is  a constant $\kappa>0$ along with  a neighborhood $V$ of $\bar x$   such that
$$
  F(x)  \subseteq F(\bar x)+\kappa \|x- \bar x\| \textbf{B}_{{\cal Y}}, \quad  \forall\,  x \in V.
$$
\end{definition}

\begin{definition}\label{Isolated calmness}
   The set-valued mapping  $F:{\cal X}\rightrightarrows {\cal Y}$ is said to be isolated  calm at  $\bar x$ for $\bar y$  if there is a constant $\kappa >0$ along with  a neighborhood $V$ of $\bar x$ and a neighborhood $W$ of $\bar y$  such that
$$
  F(x) \cap W \subseteq \{\bar y\}+\kappa \|x-\bar x\| \textbf{B}_{{\cal Y}}, \quad  \forall\,  x \in V.
$$
\end{definition}

In this section, we first study stability properties of Lagrangian generalized equilibria for the generalized Nash equilibrium problem (GNEP) defined by (\ref{eq:p1}) and (\ref{eq:coneC}) and $K_{\nu}$ is an arbitrary closed convex cone for each $\nu=1,\ldots, N$. After that we specify each $K_{\nu}$ as a polyhendral convex cone, and develop the stability properties of LGNEs.
\subsection{Stability for the general setting}\label{sub-1}
In order to express $S_{{\rm Lag}}(p)$ as a compact formulation, we introduce the following notations. Define $u=(u^1;\ldots;u^N)$, $y=(y^1;\ldots;y^N)$, $X=X_1 \times \cdots\times X_N$ and $K=K_1\times \cdots \times K_N$,
$$
\textbf{L}(x,y;p)=\left(
\begin{array}{c}
\nabla_{x^1}L_{1}(x, y^{1};w,v^{1},u^{1})\\[2pt]
\cdots\\
\nabla_{x^{N}}L_{N}(x, y^{N};w,v^{N},u^{N})
\end{array}
\right), \textbf{F}(x;w)=\left(
\begin{array}{c}
F_{1}(x,;w)\\[2pt]
\cdots\\
F_{N}(x,;w)
\end{array}
\right)
$$
and
$$
\textbf{l}(x,y)=\left(
\begin{array}{c}
\nabla_{x^1}l_{1}(x, y^{1})\\[2pt]
\cdots\\
\nabla_{x^{N}}l_{N}(x, y^{N})
\end{array}
\right), \textbf{g}(x)=\left(
\begin{array}{c}
g_{1}(x)\\[2pt]
\cdots\\
g_{N}(x)
\end{array}
\right).
$$
Then $S_{{\rm Lag}}(p)$ can be expressed as the following compact form
\begin{equation}\label{Lag-compact}
S_{\rm Lag}(p)=\left
\{(x,y)\in \mathbb R^n \times {\cal Y}:0\in \left(
\begin{array}{c}
\textbf{L}(x,y;p)\\[4pt]
-\textbf{F}(x;w)+u
\end{array}
\right)+N_{X \times K^{\circ}}(x,y)\right\}.
\end{equation}
Let $\bar p=(\bar w,0,0)$ and $(\bar x,\bar y) \in S_{{\rm Lag}}(\bar p)$.

For studying stability properties of $S_{\rm Lag}(p)$ when $p$ is perturbed from $\bar p$, we consider its linearized  set-valued mapping
\begin{equation}\label{Lag-compact}
L_{{\rm Lag}}(\delta)=\left
\{(x,y)\in \mathbb R^n \times {\cal Y}:\delta\in \left(
\begin{array}{c}
\bar q_1 +\textbf{G}x+{\cal B}y\\[4pt]
\bar q_2-{\cal A}x
\end{array}
\right)+N_{X \times K^{\circ}}(x,y)\right\}.
\end{equation}
where
\begin{equation}\label{eq:qs}
\begin{array}{l}
\bar q_1=\textbf{l}(\bar x,\bar y)-\textbf{G}\bar x-{\cal B}\bar y,\\[4pt]
\bar q_2=-\textbf{g}(\bar x)+{\cal A}\bar x
\end{array}
\end{equation}
and
\begin{equation}\label{eq:ders}
\begin{array}{l}
\textbf{G}=\left[
\begin{array}{cccc}
\nabla^2_{x^1x^1}l_1(\bar x,\bar y^1) & \nabla^2_{x^1x^2}l_1(\bar x,\bar y^1)& \cdots & \nabla^2_{x^1x^N}l_1(\bar x,\bar y^1)\\[4pt]
\nabla^2_{x^2x^1}l_2(\bar x,\bar y^2) & \nabla^2_{x^2x^2}l_2(\bar x,\bar y^2)& \cdots & \nabla^2_{x^2x^N}l_2(\bar x,\bar y^2)\\[4pt]
\vdots &\vdots & \vdots & \vdots\\[4pt]
\nabla^2_{x^Nx^1}l_N(\bar x,\bar y^N) & \nabla^2_{x^Nx^2}l_2(\bar x,\bar y^N)& \cdots & \nabla^2_{x^Nx^N}l_N(\bar x,\bar y^N)
\end{array}
\right],\\[28pt]
{\cal B}=\left(
\begin{array}{cccc}
{\rm D}_{x^1}g_1(\bar x)^* & 0 & \cdots & 0\\[4pt]
0 &D _{x^2}g_2(\bar x)^*& \cdots & 0 \\[4pt]
\vdots & \vdots & \vdots & \vdots \\[4pt]
0& 0 & \cdots & {\rm D}_{x^N}g_N(\bar x)^* \\[4pt]
\end{array}
\right),\\[28pt]
{\cal A}=\left(
\begin{array}{cccc}
{\rm D}_{x^1}g_1(\bar x) & {\rm D}_{x^2}g_1(\bar x) & \cdots &{\rm D}_{x^N}g_1(\bar x)\\[4pt]
{\rm D}_{x^1}g_2(\bar x) & {\rm D}_{x^2}g_2(\bar x) & \cdots &{\rm D}_{x^N}g_2(\bar x)\\[4pt]
\vdots & \vdots & \vdots & \vdots \\[4pt]
{\rm D}_{x^1}g_N(\bar x) & {\rm D}_{x^2}g_N(\bar x) & \cdots &{\rm D}_{x^N}g_N(\bar x) \\[4pt]
\end{array}
\right).
\end{array}
\end{equation}

The following result will be used to develop the   Lipschitz continuous single-valued localization of $S_{{\rm Lag}}$ from that of $L_{{\rm Lag}}$.
\begin{lemma}\label{first-app}
Let $f:\mathbb R^{d}\times \mathbb R^n\rightarrow \mathbb {\cal Z}$ be differentiable with respect to $x$
  in a neighborhood of $(\bar p,\bar x)$ and
let $f$ and ${\rm D}_x f$ be continuous in this neighborhood, where ${\cal Z}$ is a finite dimensional Hilbert space. Then the function $h(x) =
f(\bar p,\bar x)+{\rm D}_x f(\bar p,\bar x)(x-\bar x)$ is a strict first-order approximation to $f$ with respect to
$x$ uniformly in $p$ at $(\bar p,\bar x)$.
\end{lemma}
{\bf Proof}. Since ${\rm D}_x f$ is continuous in a neighborhood of $(\bar p, \bar x)$, for arbitrary $\varepsilon>0$ there exist a convex neighborhood
$U$ of $\bar x$ and a convex neighborhood $V$ of $\bar p$ such that
$$
\|{\rm D}_x f(p, \tilde x)-{\rm D}_x f(\bar p,\bar x)\|\leq \varepsilon \mbox{ for } p\in V \mbox{ and } \tilde x \in U.
$$
Then for $(p,x,x')\in V \times U\times U$, there exists $\tilde x$ in the segment connecting $x$ and $x'$ such that
$$
f(p,x)-f( p,x')={\rm D}_x f(p,\tilde x)(x-x')
$$
so that
$$
|f(p,x)-[f( p,x')+{\rm D}_x f( \bar p,\bar x)(x-x')]|=\|[{\rm D}_x f(p,\tilde x)-{\rm D}_x f( \bar p,\bar x)](x-x')\|
\leq \varepsilon \|x'-x\|.
$$
This implies
$$
\widehat{{\rm lip}}_x (e;(\bar p,\bar x))=0\mbox{ for }
e(p,x)=f(p,x)-[f( \bar p,\bar x)+{\rm D}_x f( \bar p,\bar x)(x-\bar x)].
$$
Therefore,  the function $h(x) =
f(\bar p,\bar x)+{\rm D}_x f(\bar p,\bar x)(x-\bar x)$ is a strict first-order approximation to $f$ with respect to
$x$ uniformly in $p$ at $(\bar p,\bar x)$.
The proof is completed. \hfill $\Box$

The following proposition shows that the  Lipschitz continuous single-valued localization of $L_{{\rm Lag}}$ is sufficient to that  of $S_{{\rm Lag}}$.
\begin{proposition}\label{single-valued-lip}
Suppose $(\bar x,\bar y)\in S_{{\rm Lag}}(\bar p)$. Let $f_{\nu}: \mathbb R^n \times \mathbb R^d\rightarrow \mathbb R$ and $F_{\nu}: \mathbb R^n \times \mathbb R^d\rightarrow {\cal Y}_{\nu}$ be  twice differentiable with respect to $x$
  in a neighborhood of $(\bar x,\bar w)$  and $f_{\nu}, F_{\nu}$,  $\nabla_x f_{\nu}$, ${\rm D}_xF_{\nu}$ and $\nabla^2_{xx} f_{\nu}$, ${\rm D}^2_{xx}F_{\nu}$ are continuous in this neighborhood. Moreover, suppose that $f_{\nu}, F_{\nu}$,  $\nabla_x f_{\nu}$, ${\rm D}_xF_{\nu}$ are Lipschitz continuous in a neighborhood around  $\bar w$ uniformly in $x$ in a neighborhood of $\bar x$.
  If $L_{{\rm Lag}}$ has a Lipschitz continuous single-valued localization at $0$ for $(\bar x,\bar y)$, then $S_{{\rm Lag}}$ has a Lipschitz continuous single-valued localization at $\bar p=(\bar w,0,0)$ for $(\bar x,\bar y)$.
  \end{proposition}
{\bf Proof}. Define
$$
f(p,x,y)=\left(
\begin{array}{c}
\textbf{L}(x,y;p)\\[4pt]
-\textbf{F}(x;w)+u
\end{array}
\right) \mbox{ and } h(x,y)=\left(
\begin{array}{c}
\bar q_1 +{\cal J}_x\textbf{L}(\bar x,\bar y;\bar p)x+{\rm D}_y\textbf{L}(\bar x,\bar y;\bar p)y\\[4pt]
\bar q_2-{\rm D}_x\textbf{F}(\bar x;\bar w,0)x
\end{array}
\right).
$$
   Noting that, for $\nu=1,\ldots, N$, $f_{\nu}: \mathbb R^n \times \mathbb R^d\rightarrow \mathbb R$ and $F_{\nu}: \mathbb R^n \times \mathbb R^d\rightarrow {\cal Y}_{\nu}$ be  twice differentiable with respect to $x$
  in a neighborhood of $(\bar x,\bar w)$  and $f_{\nu}, F_{\nu}$,  $\nabla_x f_{\nu}$, ${\rm D}_xF_{\nu}$ and $\nabla^2_{xx} f_{\nu}$, ${\rm D}^2_{xx}F_{\nu}$ are continuous in this neighborhood, from Lemma \ref{first-app}, the function $$h(x,y) =
f(\bar p,\bar x,\bar y)+{\rm D}_x f(\bar p,\bar x,\bar y)(x-\bar x)+{\rm D}_y f(\bar p,\bar x,\bar y)(y-\bar y)$$ is a strict first-order approximation to $f$ with respect to
$(x,y)$ uniformly in $p$ at $(\bar p,(\bar x,\bar y))$. And
$$
\widehat{{\rm lip}}_x (e;(\bar p,\bar x,\bar y))=0
$$
 for
 $$
e(p,x,y)=f(p,x,y)-[f( \bar p,\bar x,\bar y)+{\rm D}_x f( \bar p,\bar x,\bar y)(x-\bar x)+{\rm D}_y f( \bar p,\bar x,\bar y)(y-\bar y)].
$$
Since $f_{\nu}, F_{\nu}$,  $\nabla_x f_{\nu}$, ${\rm D}_xF_{\nu}$ are Lipschitz continuous in a neighborhood around  $\bar w$ uniformly in $x$ in a neighborhood of $\bar x$, we have
$$
\widehat{{\rm lip}}_p(f;(\bar p,(\bar x,\bar y)))\leq \lambda_0 <\infty
$$
for some $\lambda_0 >0$. Noting that
$$
S_{{\rm Lag}}(p)=\{(x,y): 0\in f(p,x,y)+N_{X \times K^{\circ}}(x,y)\} \mbox{ and } L_{{\rm Lag}}(\delta)=\{(x,y): \delta \in
h(x,y)+N_{X \times K^{\circ}}(x,y)\}.
$$

It follows from \cite[Proposition 3G.1]{DRockafellar2009}) that $(x,y) \rightarrow h(x,y)+N_{X \times K^{\circ}}(x,y)$ is strongly metrically regular at $0$ for
$(\bar x,\bar y)$ if
and only if its inverse $L_{{\rm Lag}}$ has a Lipschitz continuous single-valued localization at $0$ for $(\bar x,\bar y)$. In view of \cite[Theorem 3G.4]{DRockafellar2009}), we have if $L_{{\rm Lag}}$ has a Lipschitz continuous single-valued localization at $0$ for $(\bar x,\bar y)$, then $S_{{\rm Lag}}$ has a Lipschitz continuous single-valued localization at $\bar p=(\bar w,0,0)$ for $(\bar x,\bar y)$. The proof is completed. \hfill $\Box$

The following proposition is of considerable importance, as it establishes the equivalence between $S_{{\rm Lag}}$ and $L_{{\rm Lag}}$ in terms of both the Aubin property and the isolated calmness property.
\begin{proposition}\label{Aubin-Calm}
Suppose $(\bar x,\bar y)\in S_{{\rm Lag}}(\bar p)$. Let $f_{\nu}: \mathbb R^n \times \mathbb R^d\rightarrow \mathbb R$ and $F_{\nu}: \mathbb R^n \times \mathbb R^d\rightarrow {\cal Y}_{\nu}$ be  twice differentiable with respect to $x$
  in a neighborhood of $(\bar x,\bar w)$  and $f_{\nu}, F_{\nu}$,  $\nabla_x f_{\nu}$ and  ${\rm D}_xF_{\nu}$  are strictly differentiable with respect to $(x,w)$
   in this neighborhood.
  Then
  \begin{itemize}
  \item[{\rm (a)}] $L_{{\rm Lag}}$ has Aubin property at $0$ for $(\bar x,\bar y)$ if and only if  $S_{{\rm Lag}}$ has Aubin property at $\bar p=(\bar w,0,0)$ for $(\bar x,\bar y)$.
    \item[{\rm (b)}]
   $L_{{\rm Lag}}$ has has the isolated calmness property at $0$ for $(\bar x,\bar y)$ if and only if $S_{{\rm Lag}}$ has has the isolated calmness property at $\bar p=(\bar w,0,0)$ for $(\bar x,\bar y)$.
  \end{itemize}
\end{proposition}
{\bf Proof}. Define
$$
f(p,x,y)=\left(
\begin{array}{c}
\textbf{L}(x,y;p)\\[4pt]
-\textbf{F}(x;w)+u
\end{array}
\right) \mbox{ and } h(x,y)=\left(
\begin{array}{c}
\bar q_1 +{\cal J}_x\textbf{L}(\bar x,\bar y;\bar p)x+{\rm D}_y\textbf{L}(\bar x,\bar y;\bar p)y\\[4pt]
\bar q_2-{\rm D}_x\textbf{F}(\bar x;\bar w,0)x
\end{array}
\right).
$$
   Noting that, for $\nu=1,\ldots, N$, $f_{\nu}: \mathbb R^n \times \mathbb R^d\rightarrow \mathbb R$ and $F_{\nu}: \mathbb R^n \times \mathbb R^d\rightarrow {\cal Y}_{\nu}$ be  twice differentiable with respect to $x$
  in a neighborhood of $(\bar x,\bar w)$ and $f_{\nu}, F_{\nu}$,  $\nabla_x f_{\nu}$ and  ${\rm D}_xF_{\nu}$  are strictly differentiable with respect to $(x,w)$
   in this neighborhood. It is easy to obtain
   $$
   {\rm D}_pf(\bar p,\bar x,\bar y)=\left[
   \begin{array}{lll}
   {\rm D}_w \textbf{L}(\bar x,\bar y;\bar w,0,0) & - I_n & 0\\[4pt]
  - {\rm D}_w\textbf{F}(\bar x; \bar w) & 0 & {\cal I}_{{\cal Y}}
      \end{array}
   \right],
   $$
   which is obviously an onto operator. Noting that
$$
S_{{\rm Lag}}(p)=\{(x,y): 0\in f(p,x,y)+N_{X \times K^{\circ}}(x,y)\} \mbox{ and } L_{{\rm Lag}}(\delta)=\{(x,y): \delta \in
h(x,y)+N_{X \times K^{\circ}}(x,y)\}.
$$
It follows from \cite[Theorem 3E.6]{DRockafellar2009}) that $(x,y) \rightarrow h(x,y)+N_{X \times K^{\circ}}(x,y)$ is is   metrically regular at  $0$ for
$(\bar x,\bar y)$ if
and only if its inverse $L_{{\rm Lag}}$ has Aubin property at $0$ for $(\bar x,\bar y)$. In view of \cite[Theorem 3F.9]{DRockafellar2009}), we have that $L_{{\rm Lag}}$ has Aubin property at $0$ for $(\bar x,\bar y)$ if and only if  $S_{{\rm Lag}}$ has Aubin property at $\bar p=(\bar w,0,0)$ for $(\bar x,\bar y)$.

It follows from \cite[Theorem 3H.3]{DRockafellar2009}) that $(x,y) \rightarrow h(x,y)+N_{X \times K^{\circ}}(x,y)$ is metrically subregular at $0$ for
$(\bar x,\bar y)$ if
and only if its inverse $L_{{\rm Lag}}$ calm  at $0$ for $(\bar x,\bar y)$. In view of \cite[Theorem 3I.13]{DRockafellar2009}), we have that  $L_{{\rm Lag}}$ has has the isolated calmness property at $0$ for $(\bar x,\bar y)$ if and only if  $S_{{\rm Lag}}$ has the isolated calmness property at $\bar p=(\bar w,0,0)$ for $(\bar x,\bar y)$.
The proof is completed. \hfill $\Box$
 \begin{corollary}\label{Poly-IC}
Suppose conditions in Theorem \ref{ThIC} are satisfied and $\Theta=X \times K^{\circ}$ is polyhedral convex set, then $S_{{\rm Lag}}$ has the isolated calmness property  at $\bar p=(\bar w,0,0)$ for $(\bar x,\bar y)$ if and only if $(\bar x,\bar y)$ is
an isolated point of $L_{{\rm Lag}}(0)$.
\end{corollary}
{\bf Proof}. The result comes from Proposition \ref{Aubin-Calm} and \cite[Proposition 3I.1]{DRockafellar2009}).\hfill $\Box$

From Proposition \ref{single-valued-lip} and Proposition \ref{Aubin-Calm}, we may obtain the stability properties for $S_{{\rm Lag}}$ by analyzing the corresponding stability properties of $L_{{\rm Lag}}$ under the conditions of these propositions.
In view of the definitions of $L_{{\rm Lag}}$ and $\textbf{G}$, ${\cal B}$ and ${\cal A}$,  $L_{{\rm Lag}}$ can be  expressed as
$$
L_{{\rm Lag}}(\delta)=\left\{(x,y)\in \mathbb R^n \times {\cal Y}:
\delta \in \bar q +\left[\begin{array}{ll}
\textbf{G} &  {\cal B}\\[4pt]
-{\cal A} & 0
\end{array}
\right]\left(
\begin{array}{l}
x\\[4pt]
y
\end{array}
\right)+N_{X \times K^{\circ}}(x,y)
\right\}.
$$
Let us further denote $z=(x,y)$, $\Theta=X \times K^{\circ}$
 and
$$
\textbf{M}=\left[\begin{array}{ll}
\textbf{G} &  {\cal B}\\[4pt]
-{\cal A} & 0
\end{array}
\right].
$$
 Then $L_{{\rm Lag}}$ is expressed as of the following compact form
 \begin{equation}\label{Lcomp}
 L_{{\rm Lag}}(\delta)=\{z\in \mathbb R^n \times {\cal Y}:
\delta \in \bar q +\textbf{M}z+N_{\Theta}(z)\}.
 \end{equation}
 \begin{proposition}\label{gders}
 For the set-valued mapping defined by (\ref{Lcomp}), one has that  $L_{{\rm Lag}}$ has Aubin property at $0$ for $\bar z=(\bar x,\bar y)$ if and only if
 \begin{equation}\label{eq:c1}
 \textbf{M}^*w \in {\rm D}^*N_{\Theta}(\bar z\,| -\bar q-\textbf{M}\bar z)(-w)\Longrightarrow w=0.
 \end{equation}
  \end{proposition}
  {\bf Proof}. Let ${\cal Z}=\mathbb R^n \times {\cal Y}$ and define a linear operator $\textbf{T}: {\cal Z} \rightarrow {\cal Z}$ by
  $$
  \textbf{T}(\delta,z)=\left[
  \begin{array}{ll}
  0 & {\cal I}_{{\cal Z}}\\[4pt]
  {\cal I}_{{\cal Z}} &-\textbf{M}
  \end{array}
  \right
  ]\left(
  \begin{array}{l}
  \delta\\[4pt]
  z
  \end{array}
  \right).
  $$
  Then $T$ is nonsingular with
  $$
  \textbf{T}^*=\left[
  \begin{array}{ll}
  0 & {\cal I}_{{\cal Z}}\\[4pt]
  {\cal I}_{{\cal Z}} &-\textbf{M}^*
  \end{array}
  \right
  ]\mbox{ and } \textbf{T}^{-1}=\left[
  \begin{array}{ll}
  \textbf{M} & {\cal I}_{{\cal Z}}\\[4pt]
  {\cal I}_{{\cal Z}} &0
  \end{array}
  \right
  ].
  $$
  We may express the graph as
  $$
  {\rm gph}\,L_{{\rm Lag}}=\{(\delta, z): (z,\delta-\bar q-\textbf{M}z)\in {\rm gph}\, N_{\Theta}\}=\{(\delta, z):\textbf{T}(\delta, z)+b\in {\rm gph}\, N_{\Theta}\},
  $$
  where $b=(0;-\bar q)\in {\cal Z}$.
  Since $\textbf{T}$ is invertible, we have
  $$
  N_{{\rm gph}\,L_{{\rm Lag}}}(\delta,z)=\textbf{T}^*N_{{\rm gph}\,N_{\Theta}}(\textbf{T}(\delta,z)+b).
  $$
  From the definition of coderivative, we have
  $$
  \begin{array}{ll}
  w\in D^*L_{{\rm Lag}}(\delta\,|z)(v)& \Longleftrightarrow (w,-v)\in  N_{{\rm gph}\,L_{{\rm Lag}}}(\delta,z)\\[4pt]
  &\Longleftrightarrow (w,-v)\in\textbf{T}^*N_{{\rm gph}\,N_{\Theta}}(\textbf{T}(\delta,z)+b)\\[4pt]
  &\Longleftrightarrow \textbf{T}^{*-1}(w,-v) \in  N_{{\rm gph}\,N_{\Theta}}(\textbf{T}(\delta,z)+b)\\[4pt]
  &  \Longleftrightarrow
  \left[
  \begin{array}{ll}
  \textbf{M}^* & {\cal I}_{{\cal Z}}\\[4pt]
   {\cal I}_{{\cal Z}} & 0
   \end{array}
   \right]
  \left(
  \begin{array}{l}
  w\\[4pt]
  -v
  \end{array}
  \right) \in N_{{\rm gph}\,N_{\Theta}}(z,\delta-\bar q-\textbf{M}z)\\[4pt]
  &  \Longleftrightarrow (\textbf{M}^*w-v, w) \in N_{{\rm gph}\,N_{\Theta}}(z,\delta-\bar q-\textbf{M}z)\\[4pt]
   &  \Longleftrightarrow \textbf{M}^*w-v \in D^*N_{\Theta}(z\,|\delta-\bar q-\textbf{M}z)(-w).
  \end{array}
  $$
  It follows from Mordukhovich criterion in \cite[Theorem 9.40]{RW98} that $L_{{\rm Lag}}$ has the
Aubin property at $0$ for $\bar z=(\bar x,\bar y)$ if and only if
$$
 D^*L_{{\rm Lag}}(0\,|\bar z)(0)=\{0\}.
$$
  Thus, $L_{{\rm Lag}}$ has the
Aubin property at $0$ for $\bar z=(\bar x,\bar y)$ if and only if
$$
\textbf{M}^*w \in D^*N_{\Theta}(\bar z\,|-\bar q-\textbf{M}\bar z)(-w) \Longrightarrow w=0.
$$
  The proof is completed. \hfill $\Box$

We can now establish a necessary and sufficient condition for the Aubin property of
$S_{{\rm Lag}}$, formulated in terms of the coderivatives of the normal cone mappings
$N_X$ and $N_{K^{\circ}}$.
  \begin{theorem}\label{ThAubin}
  Suppose $(\bar x,\bar y)\in S_{{\rm Lag}}(\bar p)$. Let $f_{\nu}: \mathbb R^n \times \mathbb R^d\rightarrow \mathbb R$ and $F_{\nu}: \mathbb R^n \times \mathbb R^d\rightarrow {\cal Y}_{\nu}$ be  twice differentiable with respect to $x$
  in a neighborhood of $(\bar x,\bar w)$. Suppose that $f_{\nu}, F_{\nu}$,  $\nabla_x f_{\nu}$ and  ${\rm D}_xF_{\nu}$  are strictly differentiable with respect to $(x,w)$. Then
  $S_{{\rm Lag}}$ has Aubin property at $\bar p=(\bar w,0,0)$ for $(\bar x,\bar y)$ if and only if
  \begin{equation}\label{GNEP-Aubin}
  \left.
  \begin{array}{ll}
  \textbf{G}^*w_1-{\cal A}^*w_2 & \in D^*N_X(\bar x\,|-\bar q_1-\textbf{G}\bar x-{\cal B}\bar y)(-w_1)\\[4pt]
  {\cal B}^*w_1 \quad \quad \quad &\in D^*N_{K^{\circ}}(\bar y\,|-\bar q_2+{\cal A}\bar x)(-w_2)
  \end{array}
  \right\} \Longrightarrow w_1=0,w_2=0.
  \end{equation}
  \end{theorem}
{\bf Proof}. It follows from Proposition \ref{Aubin-Calm} and Proposition \ref{gders} that
$S_{{\rm Lag}}$ has Aubin property at $\bar p=(\bar w,0,0)$ for $(\bar x,\bar y)$ if and only if the implication (\ref{eq:c1}) holds. From
the definition of $\textbf{M}$, the expression of $\textbf{M}^*$ and $\Theta=X \times K^{\circ}$, we obtain for $w=(w_1,w_2)$ with $w_1 \in \mathbb R^n$ and $w_2 \in {\cal Y}$, the implication (\ref{eq:c1}) is equivalent to (\ref{GNEP-Aubin}). The proof is completed. \hfill $\Box$

When $X$ and $K^{\circ}$ are polyhedral, we may obtain the following interesting result about the Lipschitz continuous single-valued localization of $S_{{\rm Lag}}$.
\begin{corollary}\label{corSRegularity}

Suppose conditions in Theorem \ref{ThAubin} are satisfied and $\Theta=X \times K^{\circ}$ is polyhedral, then $S_{{\rm Lag}}$ has a Lipschitz continuous single-valued localization at $\bar p=(\bar w,0,0)$ for $(\bar x,\bar y)$ if and only if the implication (\ref{GNEP-Aubin}) holds.
\end{corollary}
{\bf Proof}. It follows from \cite[Theroem 1]{DRockafellar96} that, when $\Theta=X \times K^{\circ}$ is polyhedral, $L_{{\rm Lag}}$ has a Lipschitz continuous single-valued localization at $\bar p=(\bar w,0,0)$ if and only if $L_{{\rm Lag}}$ has Aubin property at $0$ for $\bar z=(\bar x,\bar y)$. We obtain from \cite[Proposition 2]{DRockafellar96} that $L_{{\rm Lag}}$ has a Lipschitz continuous single-valued localization at $\bar p=(\bar w,0,0)$ if and only if $S_{{\rm Lag}}$ has a Lipschitz continuous single-valued localization at $\bar p=(\bar w,0,0)$ for $(\bar x,\bar y)$. The result comes from the fact that $L_{{\rm Lag}}$ has Aubin property at $0$ for $\bar z=(\bar x,\bar y)$ if and only if the implication (\ref{GNEP-Aubin}) holds. The proof is completed. \hfill $\Box$

We now turn to analyzing the isolated calmness of
$S_{{\rm Lag}}$. To do so, we recall the expression for
$L_{{\rm Lag}}$ previously given in (\ref{Lcomp}).
\begin{proposition}\label{gders1}
 For the set-valued mapping defined by (\ref{Lcomp}), one has that  $L_{{\rm Lag}}$ has the isolated calmness property at $0$ for $\bar z=(\bar x,\bar y)$ if and only if
 \begin{equation}\label{eq:c2}
 -\textbf{M}d_z \in DN_{\Theta}(z\,|-\bar q-\textbf{M}\bar z)(d_z) \Longrightarrow d_z=0.
 \end{equation}
  \end{proposition}
{\bf Proof}.  Let ${\cal Z}=\mathbb R^n \times {\cal Y}$ and define a linear operator $\textbf{T}: {\cal Z} \rightarrow {\cal Z}$ by
  $$
  \textbf{T}=\left[
  \begin{array}{ll}
  0 & {\cal I}_{{\cal Z}}\\[4pt]
  {\cal I}_{{\cal Z}} &-\textbf{M}
  \end{array}
  \right
  ]
  $$
  Then $\textbf{T}$ is nonsingular and we may express the graph as
  $$
  {\rm gph}\,L_{{\rm Lag}}=\{(\delta, z): (z,\delta-\bar q-\textbf{M}z)\in {\rm gph}\, N_{\Theta}\}=\{(\delta, z):\textbf{T}(\delta, z)+b\in {\rm gph}\, N_{\Theta}\}
  $$
  with $b=(0;-\bar q) \in {\cal Z}$.
  Since $T$ is invertible, we have
  $$
  T_{{\rm gph}\,L_{{\rm Lag}}}(\delta,z)=\{(d_{\delta}, d_z):\textbf{T}(d_{\delta}, d_z)\in T_{{\rm gph}\,N_{\Theta}}(\textbf{T}(\delta,z)+b)\}.
  $$
  From the definition of graph derivative, we have
  $$
  \begin{array}{ll}
  d_z\in DL_{{\rm Lag}}(\delta\,|z)(d_{\delta})& \Longleftrightarrow \textbf{T}(d_{\delta}, d_z)\in  T_{{\rm gph}\,L_{{\rm Lag}}}(\delta,z)\\[4pt]
    &  \Longleftrightarrow
  \left[
  \begin{array}{ll}
  0 & {\cal I}_{{\cal Z}}\\[4pt]
   {\cal I}_{{\cal Z}} & -\textbf{M}
   \end{array}
   \right]
  \left(
  \begin{array}{l}
  d_{\delta}\\[4pt]
  d_z
  \end{array}
  \right) \in T_{{\rm gph}\,N_{\Theta}}(z,\delta-\bar q-\textbf{M}z)\\[4pt]
  &  \Longleftrightarrow (d_z,d_{\delta}-\textbf{M}d_z) \in N_{{\rm gph}\,T_{\Theta}}(z,\delta-\bar q-\textbf{M}z)\\[4pt]
   &  \Longleftrightarrow d_{\delta}-\textbf{M}d_z \in DN_{\Theta}(z\,|\delta-\bar q-\textbf{M}z)(d_z).
  \end{array}
  $$
  It follows from  
  the criterion characterizing isolated calmness  property of a set-valued mapping from  \cite{Rock92} and \cite{levy96},  
   that $L_{{\rm Lag}}$ has the
isolated calmness property  at $0$ for $\bar z=(\bar x,\bar y)$ if and only if
$$
 DL_{{\rm Lag}}(0\,|\bar z)(0)=\{0\}.
$$
  Thus, $L_{{\rm Lag}}$ has has the
isolated calmness property at $0$ for $\bar z=(\bar x,\bar y)$ if and only if
$$
-\textbf{M}d_z \in DN_{\Theta}(z\,|-\bar q-\textbf{M}\bar z)(d_z) \Longrightarrow d_z=0.
$$
  The proof is completed. \hfill $\Box$

Below  we  provide a necessary and sufficient condition of the isolated calmness property of $S_{{\rm Lag}}$ in terms of graph derivatives of normal cone  mappings $N_X$ and $N_{K^{\circ}}$.
   \begin{theorem}\label{ThIC}
  Suppose $(\bar x,\bar y)\in S_{{\rm Lag}}(\bar p)$. Let $f_{\nu}: \mathbb R^n \times \mathbb R^d\rightarrow \mathbb R$ and $F_{\nu}: \mathbb R^n \times \mathbb R^d\rightarrow {\cal Y}_{\nu}$ be  twice differentiable with respect to $x$
  in a neighborhood of $(\bar x,\bar w)$. Suppose that $f_{\nu}, F_{\nu}$,  $\nabla_x f_{\nu}$ and  ${\rm D}_xF_{\nu}$  are strictly differentiable with respect to $(x,w)$. Then
  $S_{{\rm Lag}}$ has the isolated calmness property at $\bar p=(\bar w,0,0)$ for $(\bar x,\bar y)$ if and only if
  \begin{equation}\label{GNEP-IC}
  \left.
  \begin{array}{ll}
 - \textbf{G}d_x-{\cal B}d_y & \in DN_X(\bar x\,|-\bar q_1-\textbf{G}\bar x-{\cal B}\bar y)(d_x)\\[4pt]
  {\cal A}d_x \quad \quad \quad &\in DN_{K^{\circ}}(\bar y\,|-\bar q_2+{\cal A}\bar x)(d_y)
  \end{array}
  \right\} \Longrightarrow d_x=0,d_y=0.
  \end{equation}
  \end{theorem}
  {\bf Proof}. It follows from Proposition \ref{Aubin-Calm} and Proposition \ref{gders1} that
$S_{{\rm Lag}}$ has the isolated calmness  property at $\bar p=(\bar w,0,0)$ for $(\bar x,\bar y)$ if and only if the implication (\ref{eq:c2}) holds. From
the expressions of $\textbf{M}$ and  $\Theta=X \times K^{\circ}$, we obtain for $d_z=(d_x,d_y)$ with $d_x \in \mathbb R^n$ and $d_y \in {\cal Y}$, the implication (\ref{eq:c2}) is equivalent to (\ref{GNEP-IC}). The proof is completed. \hfill $\Box$
\subsection{Equality and inequality constrained GNEPs}
In this subsection, we will apply the general stability results for GNEPs obtained in Subsection \ref{sub-1}
to equality and in equality constrained GNEPs. First of all, we consider the simplest case when all players are only constrained by equalities, namely
\begin{equation}\label{eq:eqc}
C_{\nu}(x^{-\nu})=\{x^{\nu}\in \mathbb R^{n_{\nu}}:g_{\nu}(x^{\nu},x^{-\nu})=0\},
\end{equation}
where $g:\mathbb R^n \rightarrow \mathbb R^{q_{\nu}}$ is twice continuously differentiable.

In this case $S_{{\rm Lag}}(p)$ is be expressed as the following form
\begin{equation}\label{Lag-eqC}
S_{\rm Lag}(p)=\left
\{(x,y)\in \mathbb R^n \times \mathbb R^q:\left(
\begin{array}{c}
\textbf{L}(x,y;p)\\[4pt]
-\textbf{F}(x;w)+u
\end{array}
\right)=0_{n+q}\right\},
\end{equation}
where $\textbf{F}(x;w)=(F_1(x;w),\ldots, F_N(x;w))\in \mathbb R^q$ with $q=q_1+\cdots+q_N$, and $F_{\nu}:\mathbb R^n \times \mathbb R^d \rightarrow \mathbb R^{q_{\nu}}$ satisfies $F_{\nu}(x;\bar w)=g_{\nu}(x)$.

Let $\bar x$ be a local generalized Nash equilibrium point for the GNEP problem when $C_{\nu}(x^{-\nu})$ is defined by (\ref{eq:eqc}). Suppose that
$\theta_{\nu}$ and  $g_{\nu}$  are twice continuously differentiable in a neighborhood of $\bar x$ for each $\nu=1,\ldots,N$. Suppose that ${\cal J}_{x^{\nu}}g_{\nu}(\bar x)$ is of full rank in row for each  $\nu=1,\ldots,N$, then there exists a unique vector $\bar y\in \mathbb R^q$ with
$\bar y=(\bar y^1,\ldots, \bar y^N)$, $\bar y^{\nu}\in \mathbb R^{q_{\nu}}$ such that $(\bar x,\bar y)\in S_{{\rm Lag}}(\bar p)$, where $\bar p=(\bar w,0,0)$.

On the other hand, for  $(\bar x,\bar y)\in S_{{\rm Lag}}(\bar p)$, if for any $d \in \mathbb R^n$,  satisfying
$$
{\cal J}_{x^{\nu}}g_{\nu}(\bar x)d^{\nu}=0, d^{\nu}\ne 0, \nu=1,\ldots, N,
$$
one has that
$$
\langle d^{\nu}, \nabla^2_{x^{\nu}x^{\nu}}l_{\nu}(\bar x,\bar y^{\nu})d^{\nu}\rangle >0, \nu=1,\ldots, N,
$$
then $\bar x$ be a local generalized Nash equilibrium point for the GNEP problem.

When $C_{\nu}(x^{-\nu})$ is defined by (\ref{eq:eqc}), the operators ${\cal B}$ and ${\cal A}$ in (\ref{eq:ders}) are reduced to the following matrices:
\begin{equation}\label{eq:notationsNc}
\begin{array}{l}
{\cal B}=\left(
\begin{array}{cccc}
{\cal J}_{x^1}g_1(\bar x)^T & 0 & \cdots & 0\\[4pt]
0 &{\cal J}_{x^2}g_2(\bar x)^T& \cdots & 0 \\[4pt]
\vdots & \vdots & \vdots & \vdots \\[4pt]
0& 0 & \cdots & {\cal J}_{x^N}g_N(\bar x)^T \\[4pt]
\end{array}
\right),\,\, {\cal A}=\left(
\begin{array}{c}
{\cal J} g_1(\bar x) \\[4pt]
{\cal J}g_2(\bar x) \\[4pt]
\vdots \\[4pt]
{\cal J}g_N(\bar x)
\end{array}
\right).
\end{array}
\end{equation}
\begin{proposition}\label{single-valued-lip}
  Suppose $(\bar x,\bar y)\in S_{{\rm Lag}}(\bar p)$. Let $f_{\nu}: \mathbb R^n \times \mathbb R^d\rightarrow \mathbb R$ and $F_{\nu}: \mathbb R^n \times \mathbb R^d\rightarrow \mathbb R^{q_{\nu}}$ be  twice differentiable with respect to $x$
  in a neighborhood of $(\bar x,\bar w)$. Suppose that $f_{\nu}, F_{\nu}$,  $\nabla_x f_{\nu}$ and  ${\rm D}_xF_{\nu}$  are strictly differentiable with respect to $(x,w)$.
                 Then $S_{{\rm Lag}}$ has a Lipschitz continuous single-valued localization at $\bar p=(\bar w,0,0)$ for $(\bar x,\bar y)$ if and only if the matrix
                 \begin{equation}\label{ec-matrix}
                 \left[
                 \begin{array}{ll}
                 \textbf{G}^T & -{\cal A}^T\\[4pt]
                 {\cal B}^T & 0
                 \end{array}
                 \right
                 ]
                 \end{equation}
                 is nonsingular, where $\textbf{G}$ is defined by (\ref{eq:ders}), ${\cal B}$ and ${\cal A}$ are defined by
                 (\ref{eq:notationsNc}).
  \end{proposition}
  {\bf Proof}. From the definition of $C_{\nu}(x^{-\nu})$, we have $X=\mathbb R^n$  and $K^{\circ}=\mathbb R^q$. Thus we obtain
  $$
  {\rm gph}\, N_X=\{(x, 0):x \in \mathbb R^n\}=\mathbb R^n \times \{0\}, \, {\rm gph}\, N_{K^{\circ}}=\{(y, 0):y \in \mathbb R^q\}=\mathbb R^q \times \{0\}
  $$
  and
  $$
  {\rm gph}\, N_X(x,0)=N_{\mathbb R^n \times \{0\}}(x,0)=\{0\}\times \mathbb R^n,\, {\rm gph}\, N_{K^{\circ}}(y,0)=N_{\mathbb R^q \times \{0\}}(y,0)=\{0\}\times \mathbb R^q.
  $$
  From the definition of coderivative, we obtain
  $$
   D^*N_X(x \,| 0)(w_1)=\{0\} \mbox{ and } D^*N_{K^{\circ}}(y \,| 0)(w_2)=\{0\}.
  $$
  Thus (\ref{GNEP-Aubin}) of Theorem \ref{ThAubin} is reduced to
  \begin{equation}\label{GNEP-Aubinec}
  \left.
  \begin{array}{l}
  \textbf{G}^*w_1-{\cal A}^*w_2 =0\\[4pt]
  {\cal B}^*w_1 \quad \quad \quad \,\, =0
  \end{array}
  \right\} \Longrightarrow w_1=0,w_2=0.
  \end{equation}
  The implication (\ref{GNEP-Aubinec}) is equivalent to the nonsingularity of the matrix in (\ref{ec-matrix}). Then it follows from
  Theorem \ref{ThAubin} that the nonsingularity of the matrix in (\ref{ec-matrix}) is equivalent to $S_{{\rm Lag}}$ has Aubin property at $\bar p=(\bar w,0,0)$ for $(\bar x,\bar y)$, which from Corollary \ref{corSRegularity}, is equivalent to the property that  $S_{{\rm Lag}}$ has a Lipschitz continuous single-valued localization at $\bar p=(\bar w,0,0)$ for $(\bar x,\bar y)$. The proof is completed. \hfill $\Box$

  Now we consider the case when $C_{\nu}(x^{-\nu})$ is defined by both equality and inequality constraints:
    \begin{equation}\label{eq:ineqc}
C_{\nu}(x^{-\nu})=\{x^{\nu}\in \mathbb R^{n_{\nu}}:g_{\nu}(x^{\nu},x^{-\nu})\in K_{\nu}\}\mbox{ with } K_{\nu}=\{0_{q_{\nu}}\} \times \mathbb R^{m_{\nu}-q_{\nu}}_-,
\end{equation}
  where $g_{\nu}: \mathbb R^n \rightarrow \mathbb R^{m_{\nu}}$, $\nu=1,\ldots,N$. We define $m=m_1+\ldots+m_N$.

  In this case $S_{{\rm Lag}}(p)$ is be expressed as the following form
\begin{equation}\label{Lag-eqIeqC}
S_{\rm Lag}(p)=\left
\{(x,y)\in \mathbb R^n \times \mathbb R^q:
\begin{array}{l}
0=\textbf{L}(x,y;p)\\[4pt]
0= -F_{1,a}(x;w)+u^{1,a}\\[4pt]
0\in -F_{1,b}(x;w)+u^{1,b}+N_{\mathbb R^{m_1-q_1}_+}(y^{1,b})\\[4pt]
\quad \quad \quad \vdots\\[4pt]
0=-F_{N,a}(x;w)+u^{N,a}\\[4pt]
0\in -F_{N,b}(x;w)+u^{N,b}+N_{\mathbb R^{m_N-q_N}_+}(y^{N,b})\\[4pt]
\end{array}
\right\},
\end{equation}
where
$$
\textbf{L}(x,y;p)=\left(
\begin{array}{c}
\nabla_{x^1}f_{1}(x;w)+{\cal J}_{x^1}F_1(x;w)^Ty^{1}-v^{1}\\[2pt]
\cdots\\
\nabla_{x^N}f_{N}(x;w)+{\cal J}_{x^N}F_1(x;w)^Ty^{N}-v^{N}
\end{array}
\right),\,\, F_{\nu}=\left(
\begin{array}{c}
F_{\nu,a}\\[3pt]
F_{\nu,b}
\end{array}
\right), \,\,  u^{\nu}=\left(
\begin{array}{c}
u^{\nu,a}\\[3pt]
u^{\nu,b}
\end{array}
\right),\,\, y^{\nu}=\left(
\begin{array}{c}
y^{\nu,a}\\[3pt]
y^{\nu,b}
\end{array}
\right)
$$
 with $F_{\nu,a}:\mathbb R^n \times \mathbb R^d \rightarrow \mathbb R^{q_{\nu}}$, $F_{\nu,b}:\mathbb R^n \times \mathbb R^d \rightarrow \mathbb R^{m_{\nu}-q_{\nu}}$, and $u^{\nu,a},y^{\nu,a}\in R^{q_{\nu}}$ and $u^{\nu,b}, y^{\nu,b}\in R^{m_{\nu}-q_{\nu}}$.

 Let $\bar x$ be a local generalized Nash equilibrium point for the GNEP problem when $C_{\nu}(x^{-\nu})$ is defined by (\ref{eq:ineqc}). Suppose that
$\theta_{\nu}$ and  $g_{\nu}$  are twice continuously differentiable in a neighborhood of $\bar x$ for each $\nu=1,\ldots,N$. Suppose that for each  $\nu=1,\ldots,N$,
$$
\begin{array}{l}
{\cal J}_{x^{\nu}}g_{\nu,a}(\bar x) \mbox{is of full rank in row}\\[4pt]
\mbox{ there exists a vector } \widehat d^{\nu} \in \mathbb R^{n_{\nu}}\mbox{ satisfies, and }
{\cal J}_{x^{\nu}}g_{\nu,a}(\bar x)\widehat d^{\nu}=0,\, {\cal J}_{x^{\nu}}g_{\nu,b}(\bar x)\widehat d^{\nu}<0,
\end{array}
$$
  then there exists a  vector $\bar y\in \mathbb R^m$ with
$\bar y=(\bar y^1,\ldots, \bar y^N)$, $\bar y^{\nu}\in \mathbb R^{m_{\nu}}$ such that $(\bar x,\bar y)\in S_{{\rm Lag}}(\bar p)$, where $\bar p=(\bar w,0,0)$.

For  $(\bar x,\bar y)\in S_{{\rm Lag}}(\bar p)$, define three sets of indices:
 \begin{equation}\label{3ind}
 \begin{array}{l}
 \alpha_{\nu}=\{1,\ldots, q_{\nu}\} \cup\{i: [g_{\nu}]_i(\bar x)=0< \bar y^{\nu}_i: i=q_{\nu}+1,\ldots,m_{\nu}\},\\[4pt]
 \beta_{\nu}=\{i: [g_{\nu}]_i(\bar x)=0=\bar y^{\nu}_i: i=q_{\nu}+1,\ldots,m_{\nu}\},\\[4pt]
 \gamma_{\nu}=\{i: [g_{\nu}]_i(\bar x)<0=\bar y^{\nu}_i: i=q_{\nu}+1,\ldots,m_{\nu}\}.
\end{array}
 \end{equation}
If for any $d \in \mathbb R^n$  satisfying
$$
{\cal J}_{x^{\nu}}g_{\nu,\alpha_{\nu}}(\bar x)d^{\nu}=0, {\cal J}_{x^{\nu}}g_{\nu,\beta_{\nu}}(\bar x)d^{\nu}\leq0,
d^{\nu}\ne 0, \nu=1,\ldots, N,
$$
one has that
$$
\langle d^{\nu}, \nabla^2_{x^{\nu}x^{\nu}}l_{\nu}(\bar x,\bar y^{\nu})d^{\nu}\rangle >0, \nu=1,\ldots, N,
$$
then $\bar x$ be a local generalized Nash equilibrium point for the GNEP problem.

 Let $(\bar x,\bar y) \in S_{\rm Lag}(\bar p)$ with $\bar p=(\bar w,0,0)$. For $\bar y \in N_K(\textbf{g}(\bar x))$, we have $\bar y=(\bar y^1,\ldots, \bar y^N)$ satisfies
 $\bar y^{\nu}\in N_{K_{\nu}}(g_{\nu}(\bar x))$, or $g_{\nu}(\bar x) \in N_{K_{\nu}^{\circ}}(\bar y^{\nu})$.
 One has, for $\nu=1,\ldots, N$, that $K_{\nu}^{\circ}=\mathbb R^{q_{\nu}}\times \mathbb R^{m_{\nu}-q_{\nu}}_+$ and
 \begin{equation}\label{eq:ngraph}
 \begin{array}{l}
 N_{{\rm gph}\, N_{K_{\nu}^{\circ}}}(\bar y^{\nu}, g_{\nu}(\bar x))\\[3pt]
 =N_{{\rm gph}\, N_{\mathbb R^{q_{\nu}}}}(\bar y^{\nu,a}, g_{\nu,a}(\bar x))\times N_{{\rm gph}\, N_{\mathbb R^{m_{\nu}-q_{\nu}}_+}}(\bar y^{\nu,b}, g_{\nu,b}(\bar x))\\[3pt]
 = \{0_{|\alpha_{\nu}|}\}\times \mathbb R^{|\alpha_{\nu}|}
 \times \Big[ (\mathbb R_-\times \mathbb R_+) \cup (\{0\}\times \mathbb R)\cup(\mathbb R \times \{0\})\Big]^{|\beta_{\nu}|}\times
 \mathbb R^{|\gamma_{\nu}|}\times \{0_{|\gamma_{\nu}|}\}
 \end{array}
 \end{equation}
 and
   \begin{equation}\label{eq:tgraph}
 \begin{array}{l}
 T_{{\rm gph}\, N_{K_{\nu}^{\circ}}}(\bar y^{\nu}, g_{\nu}(\bar x))\\[3pt]
 =T_{{\rm gph}\, N_{\mathbb R^{q_{\nu}}}}(\bar y^{\nu,a}, g_{\nu,a}(\bar x))\times T_{{\rm gph}\, N_{\mathbb R^{m_{\nu}-q_{\nu}}_+}}(\bar y^{\nu,b}, g_{\nu,b}(\bar x))\\[3pt]
 = \mathbb R^{|\alpha_{\nu}|}\times \{0_{|\alpha_{\nu}|}\}
 \times {\rm ghp}\, N_{\mathbb R^{|\beta_{\nu}|}_+}\times \{0_{|\gamma_{\nu}|}\}\times
 \mathbb R^{|\gamma_{\nu}|}.
 \end{array}
 \end{equation}
 \begin{theorem}\label{svalued-lip-nlp}
  Suppose $(\bar x,\bar y)\in S_{{\rm Lag}}(\bar p)$. Let $f_{\nu}: \mathbb R^n \times \mathbb R^d\rightarrow \mathbb R$ and $F_{\nu}: \mathbb R^n \times \mathbb R^d\rightarrow \mathbb R^{m_{\nu}}$ be  twice differentiable with respect to $x$
  in a neighborhood of $(\bar x,\bar w)$. Suppose that $f_{\nu}, F_{\nu}$,  $\nabla_x f_{\nu}$ and  ${\cal J}_xF_{\nu}$  are strictly differentiable with respect to $(x,w)$. Then $S_{{\rm Lag}}$ has a Lipschitz continuous single-valued localization at $\bar p=(\bar w,0,0)$ for $(\bar x,\bar y)$ if and only if, for any $\nu=1,\ldots, N$, for any partition
  $\beta_{\nu}=\beta_{\nu}^+\cup \beta_{\nu}^0\cup \beta_{\nu}^-$ and $\alpha_{\nu}^+=\alpha_{\nu}\cup \beta_{\nu}^+$,
                   \begin{equation}\label{ec-matrixc}
                   \left.
                 \begin{array}{l}
                 \textbf{G}^Tw_1=\displaystyle \sum_{\nu=1}^N \Big({\cal J}g_{\nu,\alpha_{\nu}^+}(\bar x)^T[w_2]_{\alpha_{\nu}^+}+{\cal J}g_{\nu,\beta^0_{\nu}}(\bar x)^T[w_2]_{\beta_{\nu}^0}\Big)\\[6pt]
                 {\cal J}_{x^{\nu}}g_{\nu,\alpha_{\nu}^+}(\bar x)[w_1]_{\nu}=0,\, {\cal J}_{x^{\nu}}g_{\nu,\beta_{\nu}^0}(\bar x)[w_1]_{\nu}\leq 0,
                 [w_2]_{\beta_{\nu}^0}\geq0
                 \end{array}
                 \right
                 \} \Longrightarrow w_1=0, (w_2)_{\alpha_{\nu}^+\cap\beta_{\nu}^0}=0,
                 \end{equation}
                  where $\textbf{G}$ is defined by (\ref{eq:ders}), ${\cal B}$ and ${\cal A}$ are defined by
                 (\ref{eq:notationsNc}).
  \end{theorem}
  {\bf Proof}. We apply Theorem \ref{ThAubin} to develop Aubin property of $S_{{\rm Lag}}$ at $\bar p=(\bar w,0,0)$ for $(\bar x,\bar y)$. From the definition of ${\cal B}^*$ and $K^{\circ}$, we have that the inclusion ${\cal B}^*w_1 \in D^*N_{K^{\circ}}(\bar y\,|\textbf{g}(\bar x))(-w_2)$ in (\ref{GNEP-Aubin}) is equivalent to
  $$
  {\cal J}_{x^{\nu}}g_{\nu}(\bar x)[w_1]_{\nu} \in D^*N_{K_{\nu}^{\circ}}(\bar y^{\nu}\,|g_{\nu}(\bar x)(-[w_2]_{\nu}),\, \nu=1,\ldots, N,
  $$
  where $w_2=([w_2]_1,\ldots,[w_2]_N$ and $[w_2]_{\nu}\in \mathbb R^{m_{\nu}}$ for $\nu=1,\ldots, N$.

 It follows from (\ref{eq:ngraph}) that
 \begin{equation}\label{N-Ds}
 \begin{array}{l}
  {\cal J}_{x^{\nu}}g_{\nu}(\bar x)[w_1]_{\nu} \in D^*N_{K_{\nu}^{\circ}}(\bar y^{\nu}\,|g_{\nu}(\bar x)(-[w_2]_{\nu})\\[8pt]
\Longleftrightarrow  \left\{
  \begin{array}{l}
  {\cal J}_{x^{\nu}}g_{\alpha_{\nu}}(\bar x)[w_1]_{\nu}=0\\[3pt]
  ({\cal J}_{x^{\nu}}g_{\beta_{\nu}}(\bar x)[w_1]_{\nu},[w_2]_{|\beta_{\nu}|})\in \Big[ (\mathbb R_-\times \mathbb R_+) \cup (\{0\}\times \mathbb R)\cup(\mathbb R \times \{0\})\Big]^{|\beta_{\nu}|}\\[4pt]
  [w_2]_{\gamma_{\nu}}=0
  \end{array}
  \right.
  \end{array}
   \end{equation}
   Since $X=\mathbb R^n$, we have that (\ref{GNEP-Aubin}) is reduced to
   \begin{equation}\label{eq:nlpAubin}
  \left.
  \begin{array}{l}
  \textbf{G}^*w_1-{\cal J}\textbf{g}(\bar x)^Tw_2 =0\\[4pt]
 {\cal J}_{x^{\nu}}g_{\nu}(\bar x)[w_1]_{\nu} \in D^*N_{K_{\nu}^{\circ}}(\bar y^{\nu}\,|g_{\nu}(\bar x)(-[w_2]_{\nu}),\, \nu=1,\ldots, N
  \end{array}
  \right\} \Longrightarrow w_1=0,w_2=0.
  \end{equation}
 We have from (\ref{N-Ds}) that (\ref{eq:nlpAubin}) is equivalent to
 $$
  \left.
  \begin{array}{l}
  \textbf{G}^*w_1=\displaystyle \sum_{\nu=1}^N \left({\cal J}g_{\alpha_{\nu}}(\bar x)^T[w_2]_{\alpha_{\nu}}+ {\cal J}g_{\beta_{\nu}}(\bar x)^T[w_2]_{\beta_{\nu}}\right)\\[4pt]
{\cal J}_{x^{\nu}}g_{\alpha_{\nu}}(\bar x)[w_1]_{\nu}=0\\[3pt]
  ({\cal J}_{x^{\nu}}g_{\beta_{\nu}}(\bar x)[w_1]_{\nu},[w_2]_{|\beta_{\nu}|})\in \Big[ (\mathbb R_-\times \mathbb R_+) \cup (\{0\}\times \mathbb R)\cup(\mathbb R \times \{0\})\Big]^{|\beta_{\nu}|}\\[4pt]
  \end{array}
  \right\} \Longrightarrow w_1=0,[w_2]_{\alpha_{\nu}\cup\beta_{\nu}}=0.
  $$
  Therefore,  (\ref{eq:nlpAubin}) is equivalent to that,
   for any $\nu=1,\ldots, N$, for any partition
  $\beta_{\nu}=\beta_{\nu}^+\cup \beta_{\nu}^0\cup \beta_{\nu}^-$ and $\alpha_{\nu}^+=\alpha_{\nu}\cup \beta_{\nu}^+$, the implication (\ref{ec-matrix}) holds.  Then it follows from
  Theorem \ref{ThAubin} that this condition is equivalent to $S_{{\rm Lag}}$ has Aubin property at $\bar p=(\bar w,0,0)$ for $(\bar x,\bar y)$, which from Corollary \ref{corSRegularity}, is equivalent to the property that  $S_{{\rm Lag}}$ has a Lipschitz continuous single-valued localization at $\bar p=(\bar w,0,0)$ for $(\bar x,\bar y)$. The proof is completed. \hfill $\Box$\\
  For every $\nu=1,\ldots, N$, when the strict complementarity condition holds for $\nu$-th player's problem at $(\bar x^{\nu},\bar y^{\nu})$, namely $\beta_{\nu}=\emptyset$, we have the following corollary.
  \begin{corollary}\label{cor:sc}
  Suppose $(\bar x,\bar y)\in S_{{\rm Lag}}(\bar p)$ and $\beta_{\nu}=\emptyset$ for every $\nu=1,\ldots,N$. Let $f_{\nu}: \mathbb R^n \times \mathbb R^d\rightarrow \mathbb R$ and $F_{\nu}: \mathbb R^n \times \mathbb R^d\rightarrow \mathbb R^{m_{\nu}}$ be  twice differentiable with respect to $x$
  in a neighborhood of $(\bar x,\bar w)$. Suppose that $f_{\nu}, F_{\nu}$,  $\nabla_x f_{\nu}$ and  ${\cal J}_xF_{\nu}$  are strictly differentiable with respect to $(x,w)$. Assume that $\beta_{\nu}=\emptyset$ for every $\nu=1,\ldots, N$. Then $S_{{\rm Lag}}$ has a Lipschitz continuous single-valued localization at $\bar p=(\bar w,0,0)$ for $(\bar x,\bar y)$ if and only if
  the matrix
                 \begin{equation}\label{ec-matrixSc}
                 \left[
                 \begin{array}{ll}
                 \textbf{G}^T & -{\cal A}_1^T\\[4pt]
                 {\cal B}_1^T & 0
                 \end{array}
                 \right
                 ]
                 \end{equation}
                 is nonsingular, where $\textbf{G}$ is defined by (\ref{eq:ders}), ${\cal B}_1$ and ${\cal A}_1$ are defined by
                                  \begin{equation}\label{eq:notationsNccs}
\begin{array}{l}
{\cal B}_1=\left(
\begin{array}{cccc}
{\cal J}_{x^1}g_{1,\alpha_1}(\bar x)^T & 0 & \cdots & 0\\[4pt]
0 &{\cal J}_{x^2}g_{2,\alpha_2}(\bar x)^T& \cdots & 0 \\[4pt]
\vdots & \vdots & \vdots & \vdots \\[4pt]
0& 0 & \cdots & {\cal J}_{x^N}g_{N,\alpha_N}(\bar x)^T \\[4pt]
\end{array}
\right),\,\, {\cal A}_1=\left(
\begin{array}{c}
{\cal J} g_{1,\alpha_1}(\bar x) \\[4pt]
{\cal J}g_{2,\alpha_2}(\bar x) \\[4pt]
\vdots \\[4pt]
{\cal J}g_{N,\alpha_N}(\bar x)
\end{array}
\right).
\end{array}
\end{equation}
  \end{corollary}
  \begin{remark}\label{rem:LICQ}
  In view of Corollary \ref{cor:sc}, we have that, if $\beta_{\nu}=\emptyset$ for every $\nu=1,\ldots, N$, then the Lipschitz continuous single-valued localization of $S_{{\rm Lag}}$ at $\bar p=(\bar w,0,0)$ for $(\bar x,\bar y)$ implies that the set of vectors
  $$
  \displaystyle \bigcup_{\nu=1}^N \left\{\nabla g_{\nu,i}(\bar x):i \in \alpha_{\nu}\right\}
  $$
  are linearly independent.
  \end{remark}
  Now we turn to discussing the isolated calmness property of $S_{{\rm Lag}}$  at $\bar p=(\bar w,0,0)$ for $(\bar x,\bar y)$ by using Theorem
  \ref{ThIC}.
  \begin{theorem}\label{th-ic-nlp}
  Suppose $(\bar x,\bar y)\in S_{{\rm Lag}}(\bar p)$. Let $f_{\nu}: \mathbb R^n \times \mathbb R^d\rightarrow \mathbb R$ and $F_{\nu}: \mathbb R^n \times \mathbb R^d\rightarrow \mathbb R^{m_{\nu}}$ be  twice differentiable with respect to $x$
  in a neighborhood of $(\bar x,\bar w)$. Suppose that $f_{\nu}, F_{\nu}$,  $\nabla_x f_{\nu}$ and  ${\cal J}_xF_{\nu}$  are strictly differentiable with respect to $(x,w)$. Then $S_{{\rm Lag}}$ has the isolated calmness property at $\bar p=(\bar w,0,0)$ for $(\bar x,\bar y)$ if and only if
  \begin{equation}\label{ic-nlpth}
                  \left.
                 \begin{array}{l}
                 \textbf{G}d_x+
                   \left(\begin{array}{c}
                   {\cal J}_{x^{1}}g_{1,\alpha_{1}}(\bar x)^T[d_y]_{\alpha_{1}}+{\cal J}_{x^{1}}g_{1,\beta_{1}}(\bar x)^T[d_y]_{\beta_{1}}\\[6pt]
                   \vdots\\[3pt]
                   {\cal J}_{x^{N}}g_{N,\alpha_{N}}(\bar x)^T[d_y]_{\alpha_{N}}+{\cal J}_{x^{N}}g_{N,\beta_{N}}(\bar x)^T[d_y]_{\beta_{N}}
                   \end{array}
                   \right)=0\\[6pt]
                 {\cal J}g_{\nu,\alpha_{\nu}}(\bar x)d_x=0,\nu=1,\ldots,N\\[6pt]
                  0\geq {\cal J}g_{\nu,\beta_{\nu}}(\bar x)d_x \perp [d_{y^{\nu}}]_{\beta_{\nu}}\geq 0,\nu=1,\ldots,N
                                 \end{array}
                 \right
                 \} \Longrightarrow d_x=0,[d_y]_{\alpha_{\nu}\cup\beta_{\nu}}=0,
                 \end{equation}
                  where $\textbf{G}$ is defined by (\ref{eq:ders}), ${\cal B}$ and ${\cal A}$ are defined by
                 (\ref{eq:notationsNc}).
  \end{theorem}
  {\bf Proof}. We apply Theorem \ref{ThAubin} to develop the isolated calmness property of $S_{{\rm Lag}}$ at $\bar p=(\bar w,0,0)$ for $(\bar x,\bar y)$.

  From the definition of ${\cal A}$ and $K^{\circ}$, we have that the inclusion ${\cal A}d_x \in DN_{K^{\circ}}(\bar y\,|\textbf{g}(\bar x))(d_y)$ in (\ref{GNEP-Aubin}) is equivalent to
  $$
  {\cal J}g_{\nu}(\bar x)d_x \in DN_{K_{\nu}^{\circ}}(\bar y^{\nu}\,|g_{\nu}(\bar x))(d_{y^{\nu}}),\, \nu=1,\ldots, N,
  $$
  where $d_y=(d_y^1,\ldots,d_y^N)$ and $d_y^{\nu}\in \mathbb R^{m_{\nu}}$ for $\nu=1,\ldots, N$.

 It follows from (\ref{eq:tgraph}) that
 \begin{equation}\label{N-Dsd}
 \begin{array}{l}
  {\cal J}g_{\nu}(\bar x)d_x \in DN_{K_{\nu}^{\circ}}(\bar y^{\nu}\,|g_{\nu}(\bar x))(d_{y^{\nu}})\\[8pt]
\Longleftrightarrow (d_{y^{\nu}},{\cal J}g_{\nu}(\bar x)d_x )\in T_{{\rm gph}\,N_{K_{\nu}^{\circ}}})(\bar y^{\nu}, g_{\nu}(\bar x))\\[8pt]
\Longleftrightarrow
 \left\{
  \begin{array}{l}
  {\cal J}g_{v, \alpha_{\nu}}(\bar x)d_x=0\\[3pt]
  ([d_{y^{\nu}}]_{\beta_{\nu}},  {\cal J}g_{v, \beta_{\nu}}(\bar x)d_x)\in {\rm gph}\, N_{\mathbb R^{|\beta_{\nu}|}_+}\\[4pt]
  [d_y]_{\gamma_{\nu}}=0
  \end{array}
  \right.
  \end{array}
   \end{equation}
   Since $X=\mathbb R^n$, we have that (\ref{GNEP-IC}) is reduced to
   \begin{equation}\label{nlpGNEP-IC}
  \left.
  \begin{array}{l}
 - \textbf{G}d_x-{\cal B}d_y =0\\[4pt]
  {\cal J}g_{\nu}(\bar x)d_x \in DN_{K_{\nu}^{\circ}}(\bar y^{\nu}\,|g_{\nu}(\bar x))(d_{y^{\nu}}),\nu=1,\ldots, N
  \end{array}
  \right\} \Longrightarrow d_x=0,d_y=0.
  \end{equation}
 We have from (\ref{N-Dsd}) that (\ref{nlpGNEP-IC}) is equivalent to
 $$
  \left.
  \begin{array}{l}
  \textbf{G}d_x+
                   \left(\begin{array}{c}
                   {\cal J}_{x^{1}}g_{1,\alpha_{1}}(\bar x)^T[d_y]_{\alpha_{1}}+{\cal J}_{x^{1}}g_{1,\beta_{1}}(\bar x)^T[d_y]_{\beta_{1}}\\[6pt]
                   \vdots\\[3pt]
                   {\cal J}_{x^{N}}g_{N,\alpha_{N}}(\bar x)^T[d_y]_{\alpha_{N}}+{\cal J}_{x^{N}}g_{N,\beta_{N}}(\bar x)^T[d_y]_{\beta_{N}}
                   \end{array}
                   \right)=0\\[4pt]
{\cal J}g_{\nu,\alpha_{\nu}}(\bar x)d_x=0\\[3pt]
  ([d_{y^{\nu}}]_{\beta_{\nu}},  {\cal J}g_{v, \beta_{\nu}}(\bar x)d_x)\in {\rm gph}\, N_{\mathbb R^{|\beta_{\nu}|}_+}\\[4pt]
  \end{array}
  \right\} \Longrightarrow d_x=0,[d_y]_{\alpha_{\nu}\cup\beta_{\nu}}=0.
  $$
  Therefore,  (\ref{nlpGNEP-IC}) is equivalent  to the implication (\ref{ic-nlpth}).  Then it follows from
  Theorem \ref{ThIC} that this condition is equivalent to $S_{{\rm Lag}}$ has the isolated calmness property at $\bar p=(\bar w,0,0)$ for $(\bar x,\bar y)$.  The proof is completed. \hfill $\Box$\\
\section{Consensus LGNEs for shared constrained  problems}
\setcounter{equation}{0}
In this section, we consider the case when $C_{\nu}(x^{-\nu})$ is defined by a shared constraint as follows
\begin{equation}\label{eq:coneCs}
C_{\nu}(x^{-\nu})=\{x^{\nu}\in X_{\nu}:g(x^{\nu},x^{-\nu})\in K\},
\end{equation}
where $X_{\nu}\subset \mathbb R^{n_{\nu}}$ is a nonempty closed convex set, $K \subset {\cal Y}$ is closed convex cone and ${\cal Y}$ is a finitely dimensional Hilbert space. The Lagrangian function associated with (\ref{eq:p1}) with $C_{\nu}(x^{-\nu})$  in
(\ref{eq:coneCs})
is defined by
$$
    l_{\nu}  (x^{\nu}, x^{-\nu}, y^{\nu}): = \theta_{\nu} (x^{\nu},x^{-\nu})
    +\langle y^{\nu}, g(x^{\nu},x^{-{\nu}})\rangle
$$

Consider the following generalized Nash equilibrium problem with canonically perturbations:
\begin{equation}\label{GNEP-ps}
    \textbf{P}_{\nu} (x^{-{\nu}},w,v^{\nu},u)\quad \quad\quad \quad
    \begin{array}{ll}
     \min\limits_{x^{\nu} \in X_{\nu}} & f_{\nu}  (x^{\nu}, x^{-{\nu}}; w)-\langle v^{\nu} , x^{\nu}\rangle\\[4pt]
   {\rm s.t.}& F (x^{\nu}, x^{-{\nu}}; w) - u \in K,
   \end{array}
   \end{equation}
where $w \in \mathbb R^d$, $v^{\nu} \in \mathbb{R}^{n_{\nu}}$ and $ u \in {\cal Y}$ are perturbation parameters. For simplicity, we denote
$p=(w;v^1;\ldots; v^{N};u)\in \mathbb R^d \times \mathbb R^{n_1}\times \cdots \times \mathbb R^{n_N}\times {\cal Y}$.
Let $\bar w \in \mathbb R^d$ be a fixed parameter such that
$$
f_{\nu} (x^{\nu}, x^{-{\nu}};\bar w)=\theta_{\nu}(x^{\nu}, x^{-{\nu}})\quad  \mbox{ and } \quad F (x^{\nu}, x^{-{\nu}};\bar w)=g(x^{\nu}, x^{-{\nu}}).
$$
and the Lagrangian function $l_{\nu}$ becomes
$$
    l_{\nu}  (x^{\nu}, x^{-\nu}, y^{\nu}) = f_{\nu} (x^{\nu},x^{-\nu};\bar w)
    +\langle y^{\nu}, F(x^{\nu},x^{-{\nu}};\bar w)\rangle.
$$
In this setting,  the $\nu$-th player optimization problem (\ref{GNEP-ps}) is a perturbed problem of the problem defined by (\ref{eq:p1}) and (\ref{eq:coneCs}) when
$(w,v^k,u)$ is around $(\bar w,0,0)$.

We denote the feasible set of the $\nu$-th player optimization problem (\ref{GNEP-ps}) by $C_{\nu}(x^{-\nu},w,u)$, namely
$$
C_{\nu}(x^{-\nu},w,u)=\{x^{\nu}\in X_{\nu}: F(x^{\nu}, x^{-{\nu}}; w) - u \in K\}.
$$
The Lagrangian function associated with (\ref{GNEP-ps}) is defined by
$$
    L_{\nu}  (x^{\nu}, x^{-\nu}, y^{\nu} ;w,v^{\nu},u): = f_{\nu} (x^{\nu},x^{-\nu}; w)  -\langle v^{\nu} , x^{\nu}  \rangle
    +\langle y^{\nu}, F(x^{\nu},x^{-{\nu}}; w)-u\rangle
$$
for $\nu=1,...,N$. Under the basic constraint qualification for Problem ($\textbf{P}_{\nu} (\bar x^{-{\nu}},\bar w,0,0)$)  that
$$
\left.
\begin{array}{r}
0\in N_{X_{\nu}}(\bar x^{\nu})+{\rm D}_{x^{\nu}}F(\bar x,\bar w)^*y^k=0\\[4pt]
y^k \in N_{K}(F(\bar x,\bar w))
\end{array}
\right
\} \Rightarrow y^k=0,
$$
there is the formula (c.f.  \cite[Theorem 6.14]{RW98}) that
$$
N_{C_{\nu}(\bar x^{-\nu}, \bar w,0)}(\bar x^{\nu})=N_{X_{\nu}}(\bar x^{\nu})+{\rm D}_{x^{\nu}}F(\bar x,\bar w)^*N_{K}(F(\bar x,\bar w)).
$$
The necessary optimality condition for $\bar x^{\nu}$ being a local solution of $\textbf{P}_{\nu} (\bar x^{-{\nu}},\bar w,0,0)$ is that
\begin{equation}\label{KKT-originals}
0\in \nabla_{x^{\nu}}f_{\nu}(\bar x, \bar w) \in N_{C_{\nu}(\bar x^{-\nu}, \bar w,0)}(\bar x^{\nu}).
\end{equation}
Noting that $K$ is a nonempty closed convex cone, we obtain that condition (\ref{KKT-originals}) is equivalent to
\begin{equation}\label{KKT-originalDs}
\begin{array}{l}
0\in \nabla_{x^{\nu}}L_{\nu}(\bar x, y^{\nu};\bar w,0,0)+N_{X_{\nu}}(\bar x^{\nu}),\\[4pt]
0\in -g(\bar x)+N_{K^{\circ}}(y^{\nu}),
\end{array}
\end{equation}
where $K^{\circ}$ is the polar of $K$.

Assume further that a solution $(\bar x, \bar y^1,\ldots,\bar y^N)$ of the concatenated KKT
systems satisfies $\bar y^1=\cdots=\bar y^N=\bar y$, i.e., the multipliers of the shared constraints
are equal for all players. In this case, we say that $(\bar x,\bar y)$  is a consensus Lagrangian generalized Nash equilibrium.

The requirement of a common Lagrange multiplier is a powerful equilibrium selection mechanism. It picks out a specific, economically meaningful subset of all possible LGNEs. In many applications (like market settings), the consensus LGNE is considered the most relevant solution concept.

For the consensus LGNE, it is reasonable to consider the  KKT mapping when the multipliers of the shared constraints
are equal for all players, namely the mapping  $S^o_{\rm Lag}: \mathbb{R}^d\times \mathbb R^{n}\times {\cal Y}\rightarrow \mathbb{R}^{n} \times {\cal Y}$:
\begin{equation}\label{KKT-o}
S^o_{\rm Lag}(p)=\left
\{(x,y)\in \mathbb R^n \times {\cal Y}:
\begin{array}{l}
0\in \nabla_{x^{\nu}}L_{\nu}(x, y;w,v^{\nu},u)+N_{X_{\nu}}(x^{\nu}),\nu=1,\ldots,N\\[4pt]
0\in -F(x;w)+u+N_{K^{\circ}}(y)
\end{array}
\right\}.
\end{equation}

Define
$$
\textbf{L}(x,y;p)=\left(
\begin{array}{c}
\nabla_{x^1}L_{1}(x, y;w,v^{1},u)\\[2pt]
\cdots\\
\nabla_{x^{N}}L_{N}(x, y;w,v^{N},u)
\end{array}
\right).
$$
Then $S^o_{{\rm Lag}}(p)$ can be expressed as the following compact form
\begin{equation}\label{Lag-compacts}
S^o_{\rm Lag}(p)=\left
\{(x,y)\in \mathbb R^n \times {\cal Y}:0\in \left(
\begin{array}{c}
\textbf{L}(x,y;p)\\[4pt]
-F(x;w)+u
\end{array}
\right)+N_{X \times K^{\circ}}(x,y)\right\}.
\end{equation}
Let $\bar p=(\bar w,0,0)$ and $(\bar x,\bar y) \in S^o_{{\rm Lag}}(\bar p)$. Consider the following set-valued mapping
\begin{equation}\label{Lag-compacts}
L^o_{{\rm Lag}}(\delta)=\left
\{(x,y)\in \mathbb R^n \times {\cal Y}:\delta\in \left(
\begin{array}{c}
\bar q_1 +\textbf{G}x+{\rm D}g(\bar x)^*y\\[4pt]
\bar q_2-{\rm D}g(\bar x)x
\end{array}
\right)+N_{X \times K^{\circ}}(x,y)\right\}.
\end{equation}
where
\begin{equation}\label{eq:qss}
\begin{array}{l}
\bar q_1=\textbf{l}(\bar x,\bar y)-\textbf{G}\bar x-{\rm D}g(\bar x)^*\bar y,\\[4pt]
\bar q_2=-g(\bar x)+{\rm D}g(\bar x)\bar x
\end{array}
\end{equation}
and
\begin{equation}\label{eq:derss}
\begin{array}{l}
\textbf{G}=\left[
\begin{array}{cccc}
\nabla^2_{x^1x^1}l_1(\bar x,\bar y) & \nabla^2_{x^1x^2}l_1(\bar x,\bar y)& \cdots & \nabla^2_{x^1x^N}l_1(\bar x,\bar y)\\[4pt]
\nabla^2_{x^2x^1}l_2(\bar x,\bar y) & \nabla^2_{x^2x^2}l_2(\bar x,\bar y)& \cdots & \nabla^2_{x^2x^N}l_2(\bar x,\bar y)\\[4pt]
\vdots &\vdots & \vdots & \vdots\\[4pt]
\nabla^2_{x^Nx^1}l_N(\bar x,\bar y) & \nabla^2_{x^Nx^2}l_2(\bar x,\bar y)& \cdots & \nabla^2_{x^Nx^N}l_N(\bar x,\bar y)
\end{array}
\right],\, l_{\nu}(x,y)=\theta_{\nu}(x)+\langle y, g(x)\rangle.
\end{array}
\end{equation}
Just like Theorem \ref{ThAubin}, we can easily obtain the following result.
\begin{theorem}\label{ThAubins}
  Suppose $(\bar x,\bar y)\in S_{{\rm Lag}}(\bar p)$. Let $f_{\nu}: \mathbb R^n \times \mathbb R^d\rightarrow \mathbb R$ and $F: \mathbb R^n \times \mathbb R^d\rightarrow {\cal Y}$ be  twice differentiable with respect to $x$
  in a neighborhood of $(\bar x,\bar w)$. Suppose that $f_{\nu}, F$,  $\nabla_x f_{\nu}$ and  ${\rm D}_xF$  are strictly differentiable with respect to $(x,w)$. Then
  $S^o_{{\rm Lag}}$ has Aubin property at $\bar p=(\bar w,0,0)$ for $(\bar x,\bar y)$ if and only if
  \begin{equation}\label{GNEP-Aubins}
  \left.
  \begin{array}{ll}
  \textbf{G}^*w_1-{\rm D}g(\bar x)^*w_2 & \in D^*N_X(\bar x\,|-\bar q_1-\textbf{G}\bar x-{\rm D}g(\bar x)^*\bar y)(-w_1)\\[4pt]
  {\rm D}g(\bar x)w_1 \quad \quad \quad &\in D^*N_{K^{\circ}}(\bar y\,|g(\bar x))(-w_2)
  \end{array}
  \right\} \Longrightarrow w_1=0,w_2=0.
  \end{equation}
  \end{theorem}
  Just like Theorem \ref{ThIC}, we can easily obtain the following result.
  \begin{theorem}\label{ThICs}
  Suppose $(\bar x,\bar y)\in S_{{\rm Lag}}(\bar p)$. Let $f_{\nu}: \mathbb R^n \times \mathbb R^d\rightarrow \mathbb R$ and $F: \mathbb R^n \times \mathbb R^d\rightarrow {\cal Y}$ be  twice differentiable with respect to $x$
  in a neighborhood of $(\bar x,\bar w)$. Suppose that $f_{\nu}, F$,  $\nabla_x f_{\nu}$ and  ${\rm D}_xF$  are strictly differentiable with respect to $(x,w)$. Then
  $S^o_{{\rm Lag}}$ has the isolated calmness property at $\bar p=(\bar w,0,0)$ for $(\bar x,\bar y)$ if and only if
  \begin{equation}\label{GNEP-ICs}
  \left.
  \begin{array}{ll}
 - \textbf{G}d_x-{\rm D}g(\bar x)^*d_y & \in DN_X(\bar x\,|-\bar q_1-\textbf{G}\bar x-{\rm D}g(\bar x)^*\bar y)(d_x)\\[4pt]
  {\rm D}g(\bar x)d_x \quad \quad \quad &\in DN_{K^{\circ}}(\bar y\,|g(\bar x))(d_y)
  \end{array}
  \right\} \Longrightarrow d_x=0,d_y=0.
  \end{equation}
  \end{theorem}

 Now we consider the case when $C_{\nu}(x^{-\nu})$ is defined by  shared equality and inequality constraints:
    \begin{equation}\label{eq:ineqcs}
C_{\nu}(x^{-\nu})=\{x^{\nu}\in \mathbb R^{n_{\nu}}:g(x^{\nu},x^{-\nu})\in K\}\mbox{ with } K=\{0_{q}\} \times \mathbb R^{m-q}_-,
\end{equation}
  where $g: \mathbb R^n \rightarrow \mathbb R^{m}$.

  In this case $S^o_{{\rm Lag}}(p)$ is be expressed as the following form
\begin{equation}\label{Lag-eqIeqCs}
S^o_{\rm Lag}(p)=\left
\{(x,y)\in \mathbb R^n \times \mathbb R^q:
\begin{array}{l}
0=\textbf{L}(x,y;p)\\[4pt]
0= -F_{a}(x;w)+u^{a}\\[4pt]
0\in -F_{b}(x;w)+u^{b}+N_{\mathbb R^{m-q}_+}(y^{b})
\end{array}
\right\},
\end{equation}
where
$$
\textbf{L}(x,y;p)=\left(
\begin{array}{c}
\nabla_{x^1}f_{1}(x;w)+{\cal J}_{x^1}F(x;w)^Ty-v^{1}\\[2pt]
\cdots\\
\nabla_{x^N}f_{N}(x;w)+{\cal J}_{x^N}F(x;w)^Ty-v^{N}
\end{array}
\right),\,\, F=\left(
\begin{array}{c}
F_{a}\\[3pt]
F_{b}
\end{array}
\right), \,\,  u=\left(
\begin{array}{c}
u^{a}\\[3pt]
u^{b}
\end{array}
\right),\,\, y=\left(
\begin{array}{c}
y^{a}\\[3pt]
y^{b}
\end{array}
\right)
$$
 with $F_{a}:\mathbb R^n \times \mathbb R^d \rightarrow \mathbb R^{q}$, $F_{b}:\mathbb R^n \times \mathbb R^d \rightarrow \mathbb R^{m-q}$, and $u^{a},y^{a}\in \mathbb R^{q}$ and $u^{b}, y^{b}\in \mathbb R^{m-q}$.

For  $(\bar x,\bar y)\in S^o_{{\rm Lag}}(\bar p)$, define three sets of indices:
 \begin{equation}\label{3inds}
 \begin{array}{l}
 \alpha=\{1,\ldots, q\} \cup\{i: g_i(\bar x)=0< \bar y_i: i=q+1,\ldots,m\},\\[4pt]
 \beta=\{i: g_i(\bar x)=0=\bar y_i: i=q+1,\ldots,m\},\\[4pt]
 \gamma=\{i: g_i(\bar x)<0=\bar y_i: i=q+1,\ldots,m\}.
\end{array}
 \end{equation}
If for any $d \in \mathbb R^n$  satisfying
$$
{\cal J}_{x^{\nu}}g_{\alpha}(\bar x)d^{\nu}=0, {\cal J}_{x^{\nu}}g_{\beta}(\bar x)d^{\nu}\leq0,
d^{\nu}\ne 0, \nu=1,\ldots, N,
$$
one has that
$$
\langle d^{\nu}, \nabla^2_{x^{\nu}x^{\nu}}l_{\nu}(\bar x,\bar y)d^{\nu}\rangle >0, \nu=1,\ldots, N,
$$
then $\bar x$ be a local generalized Nash equilibrium point for the GNEP problem.

 Let $(\bar x,\bar y) \in S^o_{\rm Lag}(\bar p)$ with $\bar p=(\bar w,0,0)$. Then we have that $\bar y \in \mathbb R^m$ satisfies
 $\bar y\in N_{K}(g(\bar x))$, or $g(\bar x) \in N_{K^{\circ}}(\bar y)$.
 One has, for $K=\{0_q\}\times \mathbb R^{m-q}_-$, $K^{\circ}=\mathbb R^{q}\times \mathbb R^{m-q}_+$ and
 \begin{equation}\label{eq:ngraphs}
 \begin{array}{l}
 N_{{\rm gph}\, N_{K^{\circ}}}(\bar y, g(\bar x))\\[3pt]
 =N_{{\rm gph}\, N_{\mathbb R^q}}(\bar y^{a}, g_{a}(\bar x))\times N_{{\rm gph}\, N_{\mathbb R^{m-q}_+}}(\bar y^{b}, g_{b}(\bar x))\\[3pt]
 = \{0_{|\alpha|}\}\times \mathbb R^{|\alpha|}
 \times \Big[ (\mathbb R_-\times \mathbb R_+) \cup (\{0\}\times \mathbb R)\cup(\mathbb R \times \{0\})\Big]^{|\beta|}\times
 \mathbb R^{|\gamma|}\times \{0_{|\gamma|}\}
 \end{array}
 \end{equation}
 and
   \begin{equation}\label{eq:tgraphs}
 \begin{array}{l}
 T_{{\rm gph}\, N_{K^{\circ}}}(\bar y, g(\bar x))\\[3pt]
 =T_{{\rm gph}\, N_{\mathbb R^{q}}}(\bar y^{a}, g_{a}(\bar x))\times T_{{\rm gph}\, N_{\mathbb R^{m-q}_+}}(\bar y^{b}, g_{b}(\bar x))\\[3pt]
 = \mathbb R^{|\alpha|}\times \{0_{|\alpha|}\}
 \times {\rm ghp}\, N_{\mathbb R^{|\beta|}_+}\times \{0_{|\gamma|}\}\times
 \mathbb R^{|\gamma|}.
 \end{array}
 \end{equation}
 \begin{theorem}\label{svalued-lip-nlps}
  Suppose $(\bar x,\bar y)\in S^o_{{\rm Lag}}(\bar p)$. Let $f_{\nu}: \mathbb R^n \times \mathbb R^d\rightarrow \mathbb R$ and $F: \mathbb R^n \times \mathbb R^d\rightarrow \mathbb R^m$ be  twice differentiable with respect to $x$
  in a neighborhood of $(\bar x,\bar w)$. Suppose that $f_{\nu}, F$,  $\nabla_x f_{\nu}$ and  ${\cal J}_xF$  are strictly differentiable with respect to $(x,w)$. Then $S^o_{{\rm Lag}}$ has a Lipschitz continuous single-valued localization at $\bar p=(\bar w,0,0)$ for $(\bar x,\bar y)$ if and only if, for any partition
  $\beta=\beta^+\cup \beta^0\cup \beta^-$ and $\alpha^+=\alpha\cup \beta^+$,
                   \begin{equation}\label{ec-matrixcs}
                   \left.
                 \begin{array}{l}
                 \textbf{G}^Tw_1={\cal J}g_{\alpha^+}(\bar x)^T[w_2]_{\alpha^+}+{\cal J}g_{\beta^0}(\bar x)^T[w_2]_{\beta^0}\\[6pt]
                 {\cal J}g_{\alpha^+}(\bar x)w_1=0,\, {\cal J}g_{\beta^0}(\bar x)w_1\leq 0,
                 [w_2]_{\beta^0}\geq0
                 \end{array}
                 \right
                 \} \Longrightarrow w_1=0, [w_2]_{\alpha^+\cup\beta^0}=0,
                 \end{equation}
                  where $\textbf{G}$ is defined by (\ref{eq:derss}).
  \end{theorem}
  {\bf Proof}. We apply Theorem \ref{ThAubins} to develop Aubin property of $S^o_{{\rm Lag}}$ at $\bar p=(\bar w,0,0)$ for $(\bar x,\bar y)$. From (\ref{eq:ngraphs}), we have that the inclusion ${\rm D}g(\bar x)w_1 \in D^*N_{K^{\circ}}(\bar y\,|g(\bar x))(-w_2)$ in (\ref{GNEP-Aubins}) is equivalent to

 \begin{equation}\label{N-Dss}
 \begin{array}{l}
     {\cal J}g_{\alpha}(\bar x)w_1=0\\[3pt]
  ({\cal J}g_{\beta}(\bar x)w_1,[w_2]_{\beta})\in \Big[ (\mathbb R_-\times \mathbb R_+) \cup (\{0\}\times \mathbb R)\cup(\mathbb R \times \{0\})\Big]^{|\beta_{\nu}|}\\[4pt]
  [w_2]_{\gamma}=0
  \end{array}
       \end{equation}
   Since $X=\mathbb R^n$, we have that (\ref{GNEP-Aubins}) is reduced to
   \begin{equation}\label{eq:nlpAubinss}
  \left.
  \begin{array}{l}
  \textbf{G}^*w_1-{\cal J}g(\bar x)^Tw_2 =0\\[4pt]
 {\cal J}g(\bar x)w_1 \in D^*N_{K^{\circ}}(\bar y\,|g(\bar x)(-w_2)
  \end{array}
  \right\} \Longrightarrow w_1=0,w_2=0.
  \end{equation}
 We have from (\ref{N-Dss}) that (\ref{eq:nlpAubinss}) is equivalent to
 $$
  \left.
  \begin{array}{l}
  \textbf{G}^*w_1={\cal J} g_{\alpha}(\bar x)^T[w_2]_{\alpha}+ {\cal J} g_{\beta}(\bar x)^T[w_2]_{\beta}\\[4pt]
{\cal J} g_{\alpha}(\bar x)w_1=0\\[3pt]
  ({\cal J}g_{\beta}(\bar x)w_1,[w_2]_{\beta})\in \Big[ (\mathbb R_-\times \mathbb R_+) \cup (\{0\}\times \mathbb R)\cup(\mathbb R \times \{0\})\Big]^{|\beta|}\\[4pt]
  \end{array}
  \right\} \Longrightarrow w_1=0,[w_2]_{\alpha \cup\beta}=0.
  $$
  Therefore,  (\ref{eq:nlpAubinss}) is equivalent to that,
   for any partition
  $\beta=\beta^+\cup \beta^0\cup \beta^-$ and $\alpha^+=\alpha\cup \beta^+$, the implication (\ref{ec-matrixcs}) holds.  Then it follows from
  Theorem \ref{ThAubins} that this condition is equivalent to $S_{{\rm Lag}}$ has Aubin property at $\bar p=(\bar w,0,0)$ for $(\bar x,\bar y)$, which from Corollary \ref{corSRegularity}, is equivalent to the property that  $S^o_{{\rm Lag}}$ has a Lipschitz continuous single-valued localization at $\bar p=(\bar w,0,0)$ for $(\bar x,\bar y)$. The proof is completed. \hfill $\Box$
  \begin{corollary}\label{cor:scs}
  Suppose $(\bar x,\bar y)\in S^o_{{\rm Lag}}(\bar p)$ and $\beta_{\nu}=\emptyset$ for every $\nu=1,\ldots,N$. Let $f_{\nu}: \mathbb R^n \times \mathbb R^d\rightarrow \mathbb R$ and $F: \mathbb R^n \times \mathbb R^d\rightarrow \mathbb R^m$ be  twice differentiable with respect to $x$
  in a neighborhood of $(\bar x,\bar w)$. Suppose that $f_{\nu}, F$,  $\nabla_x f_{\nu}$ and  ${\cal J}_xF$  are strictly differentiable with respect to $(x,w)$. Then $S^o_{{\rm Lag}}$ has a Lipschitz continuous single-valued localization at $\bar p=(\bar w,0,0)$ for $(\bar x,\bar y)$ if and only if
  the matrix
                 \begin{equation}\label{ec-matrixSc}
                 \left[
                 \begin{array}{ll}
                 \textbf{G}^T & -{\cal J}g_{\alpha}(\bar x)^T\\[4pt]
                 {\cal J}g_{\alpha}(\bar x) & 0
                 \end{array}
                 \right
                 ]
                 \end{equation}
                 is nonsingular, where $\textbf{G}$ is defined by (\ref{eq:derss}).
  \end{corollary}
  Now, we propose a set of second-order sufficient optimality conditions for a  vector $\bar x$ being a generalized Nash equilibrium point.
Suppose $S^o_{{\rm Lag}}(\bar p)\ne \emptyset$, let  $(\bar x,\bar y)\in S^o_{{\rm Lag}}(\bar p)$, define
  $$
  {\cal C}(\bar x,\bar y)=\{d\in \mathbb R^n:  {\cal J}g_{\alpha}(\bar x)d=0,  {\cal J}g_{\beta}(\bar x)d\leq 0\}.
  $$
  We propose the following second-order optimality condition:
  \begin{equation}\label{2ndconds}
  d \in {\cal C}(\bar x,\bar y), d\ne 0 \Longrightarrow \langle d, \textbf{G}^Td \rangle>0.
  \end{equation}
  \begin{proposition}\label{a-suff}
  Let $f_{\nu}: \mathbb R^n \times \mathbb R^d\rightarrow \mathbb R$ and $F: \mathbb R^n \times \mathbb R^d\rightarrow \mathbb R^m$ be  twice differentiable with respect to $x$
  in a neighborhood of $(\bar x,\bar w)$. Suppose that $f_{\nu}, F$,  $\nabla_x f_{\nu}$ and  ${\cal J}_xF$  are strictly differentiable with respect to $(x,w)$. Suppose $S^o_{{\rm Lag}}(\bar p)\ne \emptyset$ and let  $(\bar x,\bar y)\in S^o_{{\rm Lag}}(\bar p)$. If the second-order optimality condition (\ref{2ndconds}) holds, then $\bar x$ is a local generalized Nash equilibrium point.
  \end{proposition}
  {\bf Proof}. For $\nu$-player, since $(\bar x,\bar y)\in S^o_{{\rm Lag}}(\bar p)$, one has that
  $$
  \nabla_{x^{\nu}}l_{\nu}(\bar x, \bar y)=0,\,\, \bar y \in N_{K}g(\bar x),
  $$
  which implies the KKT conditions hold at $(\bar x^{\nu}, \bar y)$ for  $\nu$-player. For any
  $d^{\nu} \in \mathbb R^{\nu}$, $d^{\nu}\ne 0$, satisfying
  $$
  {\cal J}_{x^{\nu}}g_{\alpha}(\bar x)d^{\nu}=0,\,\, {\cal J}_{x^{\nu}}g_{\beta}(\bar x)d^{\nu}\leq0,
  $$
  we have that the vector $\tilde d \in \mathbb R^n$ defined by $\tilde d^{\nu}=d^{\nu}$ and $\tilde d^{-\nu}=0$ satisfies $\tilde d \in {\cal C}(\bar x,\bar y)$ with $\tilde d \ne 0$. Then it follows from (\ref{2ndconds}) that
  $$
  \tilde d^T\textbf{G}^T\tilde d>0,
  $$
  which implies $\langle d^{\nu}, \nabla^2_{x^{\nu}x^{\nu}}l(\bar x,\bar y)d^{\nu}\rangle >0$. This implies the second-order sufficiency optimality conditions for problem $\textbf{P}_{\nu} (\bar{x}^{-\nu},\bar w, 0,0)$ at $(\bar x^{\nu},\bar y)$ and $\bar x$ is a strict local minimum point of
  problem $\textbf{P}_{\nu} (\bar{x}^{-\nu},\bar w, 0,0)$. This shows that $\bar x$ is  local generalized Nash equilibrium point.\hfill $\Box$\\
  Interestingly, we may use the second-order optimality condition (\ref{2ndconds}) to derive the Lipschitz continuous single-valued localization of $S^o_{{\rm Lag}}$ at $\bar p=(\bar w,0,0)$ for $(\bar x,\bar y)$, which is described in the next proposition.

  \begin{proposition}\label{a-suff-ic}
  Let $f_{\nu}: \mathbb R^n \times \mathbb R^d\rightarrow \mathbb R$ and $F: \mathbb R^n \times \mathbb R^d\rightarrow \mathbb R^m$ be  twice differentiable with respect to $x$
  in a neighborhood of $(\bar x,\bar w)$. Suppose that $f_{\nu}, F$,  $\nabla_x f_{\nu}$ and  ${\cal J}_xF$  are strictly differentiable with respect to $(x,w)$. Suppose $S^o_{{\rm Lag}}(\bar p)\ne \emptyset$ and let  $(\bar x,\bar y)\in S^o_{{\rm Lag}}(\bar p)$. If
   \begin{itemize}
   \item[{\rm (i)}]the set of vectors $\{\nabla g_i(\bar x):i \in \alpha \cup \beta\}$ are linearly independent;
   \item[{\rm (ii)}]the second-order optimality condition (\ref{2ndconds}) holds.
   \end{itemize}
   Then $\bar x$ is a local generalized Nash equilibrium point and $S^o_{{\rm Lag}}$ has a Lipschitz continuous single-valued localization at $\bar p=(\bar w,0,0)$ for $(\bar x,\bar y)$.
  \end{proposition}
  {\bf Proof}. We only need to prove that $S^o_{{\rm Lag}}$ has a Lipschitz continuous single-valued localization at $\bar p=(\bar w,0,0)$ for $(\bar x,\bar y)$ if (i) and (ii) are satisfied. For doing this, we verify the implication (\ref{ec-matrixcs}) for
    any partition
  $\beta=\beta^+\cup \beta^0\cup \beta^-$ and $\alpha^+=\alpha\cup \beta^+$. Let $\beta=\beta^+\cup \beta^0\cup \beta^-$ be a partition of
  $\beta$, consider the system of inequalities
  \begin{equation}\label{eq:kkts}
  \begin{array}{l}
                 \textbf{G}^Tw_1={\cal J}g_{\alpha^+}(\bar x)^T[w_2]_{\alpha^+}+{\cal J}g_{\beta^0}(\bar x)^T[w_2]_{\beta^0}\\[6pt]
                 {\cal J}g_{\alpha^+}(\bar x)w_1=0,\, {\cal J}g_{\beta^0}(\bar x)w_1\leq 0,
                 [w_2]_{\beta^0}\geq0.
                 \end{array}
  \end{equation}
  The last two inequalities of (\ref{eq:kkts}) imply $w_1 \in {\cal C}(\bar x,\bar y)$.
  Premultiplying $w_1^T$ to the first equality of (\ref{eq:kkts}) yields
  $$
  \langle w_1, \textbf{G}^Tw_1\rangle =\langle w_1, {\cal J}g_{\alpha^+}(\bar x)^T[w_2]_{\alpha^+}+{\cal J}g_{\beta^0}(\bar x)^T[w_2]_{\beta^0}\rangle=\langle w_1, {\cal J}g_{\beta^0}(\bar x)^T[w_2]_{\beta^0}\rangle
  $$
  Since the value on the left-hand side is nonnegative and the value on the right-hand side is nonpositive, from condition (ii), we obtain $w_1=0$.
  Thus we obtain from the first equation of (\ref{eq:kkts}) that
  $$
   0={\cal J}g_{\alpha^+}(\bar x)^T[w_2]_{\alpha^+}+{\cal J}g_{\beta^0}(\bar x)^T[w_2]_{\beta^0}.
                 $$
  From (i), we know that the set of vectors $\{\nabla g_i(\bar x):i \in \alpha \cup \beta^0\}$ are linearly independent, and the above equality implies $[w_2]_{\alpha^+\cup \beta^0}=0$. Thus the implication (\ref{ec-matrixcs}) holds
   for any partition
  $\beta=\beta^+\cup \beta^0\cup \beta^-$ and $\alpha^+=\alpha\cup \beta^+$. \hfill $\Box$

  Now we turn to discussing the isolated calmness property of $S^o_{{\rm Lag}}$  at $\bar p=(\bar w,0,0)$ for $(\bar x,\bar y)$ by using Theorem
  \ref{ThICs}.
  \begin{theorem}\label{th-ic-nlps}
 Suppose $(\bar x,\bar y)\in S^o_{{\rm Lag}}(\bar p)$. Let $f_{\nu}: \mathbb R^n \times \mathbb R^d\rightarrow \mathbb R$ and $F: \mathbb R^n \times \mathbb R^d\rightarrow \mathbb R^m$ be  twice differentiable with respect to $x$
  in a neighborhood of $(\bar x,\bar w)$. Suppose that $f_{\nu}, F$,  $\nabla_x f_{\nu}$ and  ${\cal J}_xF$  are strictly differentiable with respect to $(x,w)$.  Then $S^o_{{\rm Lag}}$ has the isolated calmness property at $\bar p=(\bar w,0,0)$ for $(\bar x,\bar y)$ if and only if
  \begin{equation}\label{ic-nlpths}
                  \left.
                 \begin{array}{l}
                 \textbf{G}d_x
                   +{\cal J}g_{\alpha}(\bar x)^T[d_y]_{\alpha}+{\cal J}g_{\beta}(\bar x)^T[d_y]_{\beta}=0\\[6pt]
                 {\cal J}g_{\alpha}(\bar x)d_x=0,\\[6pt]
                  0\geq {\cal J}g_{\beta}(\bar x)d_x \perp [d_y]_{\beta}\geq 0,\nu=1,\ldots,N
                                 \end{array}
                 \right
                 \} \Longrightarrow d_x=0,[d_y]_{\alpha\cup\beta}=0,
                 \end{equation}
                  where $\textbf{G}$ is defined by (\ref{eq:derss}).
  \end{theorem}
  {\bf Proof}. We apply Theorem \ref{ThICs} to develop the isolated calmness property of $S^o_{{\rm Lag}}$ at $\bar p=(\bar w,0,0)$ for $(\bar x,\bar y)$.
In view of (\ref{eq:tgraphs}),  the inclusion ${\cal J}g(\bar x)d_x \in DN_{K^{\circ}}(\bar y\,|g(\bar x))(d_y)$ in (\ref{GNEP-ICs}) is equivalent to
  $$
 (d_y, {\cal J}g(\bar x)d_x) \in T_{N_{K^{\circ}}}(\bar y\,|g(\bar x)).
  $$
  Thus,  from (\ref{eq:tgraphs}), we have
\begin{equation}\label{N-Dsds}
(d_y, {\cal J}g(\bar x)d_x) \in T_{N_{K^{\circ}}}(\bar y\,|g(\bar x))\Longleftrightarrow \left\{
  \begin{array}{l}
  {\cal J}g_{\alpha}(\bar x)d_x=0\\[3pt]
  ([d_y]_{\beta},  {\cal J}g_{\beta}(\bar x)d_x)\in {\rm gph}\, N_{\mathbb R^{|\beta|}_+}\\[4pt]
  [d_y]_{\gamma}=0
  \end{array}
  \right.
\end{equation}
   Since $X=\mathbb R^n$, we have that (\ref{GNEP-ICs}) is reduced to
   \begin{equation}\label{nlpGNEP-ICss}
  \left.
  \begin{array}{l}
 - \textbf{G}d_x-{\cal J}g(\bar x)^Td_y =0\\[4pt]
  {\cal J}g(\bar x)d_x \in DN_{K^{\circ}}(\bar y\,|g(\bar x))(d_y)
  \end{array}
  \right\} \Longrightarrow d_x=0,d_y=0.
  \end{equation}
 We have from (\ref{N-Dsds}) that (\ref{nlpGNEP-ICss}) is equivalent to
 $$
  \left.
  \begin{array}{l}
  \textbf{G}d_x+{\cal J}g_{\alpha}g(\bar x)^T[dy]_{\alpha}+{\cal J}g_{\beta}(\bar x)^T[dy]_{\beta}=0\\[4pt]
{\cal J}g_{\alpha}(\bar x)d_x=0\\[3pt]
  ([d_y]_{\beta},  {\cal J}g_{\beta}(\bar x)d_x)\in {\rm gph}\, N_{\mathbb R^{|\beta|}_+}\\[4pt]
  \end{array}
  \right\} \Longrightarrow d_x=0,[d_y]_{\alpha\cup\beta}=0.
  $$
  Therefore,  (\ref{ThICs}) is equivalent  to the implication (\ref{ic-nlpths}).  Then it follows from
  Theorem \ref{ThICs} that this condition is equivalent
   to the isolated calmness property of $S^o_{{\rm Lag}}$ at $\bar p=(\bar w,0,0)$ for $(\bar x,\bar y)$.  The proof is completed. \hfill $\Box$\\
  \section{Classical Nash equilibrium problems}
  \setcounter{equation}{0}
  In this section, we apply the results in Subsection \ref{sub-1} to the classical Nash equilibrium problem (NEP).
  For NEP,  given the other players' strategies $x^{-\nu}$,  player $\nu$,
  solves the minimization problem
\begin{equation}\label{eq:p1nc}
\begin{array}{ll}
\min_{x_{\nu}}& \theta_{\nu}(x^{\nu},x^{-\nu})\\[4pt]
{\rm s.t.} & x_{\nu}\in C_{\nu},
\end{array}
\end{equation}
where  $C_{\nu}$ is specified as
\begin{equation}\label{eq:coneCnc}
C_{\nu}=\{x^{\nu}\in X_{\nu}:g_{\nu}(x^{\nu})\in K_{\nu}\},
\end{equation}
where $X_{\nu}\subset \mathbb R^{n_{\nu}}$ is a nonempty closed convex set, $K_{\nu}\subset {\cal Y}_{\nu}$ is closed convex cone and ${\cal Y}_{\nu}$ is a finitely dimensional Hilbert space. The Lagrangian function associated with (\ref{eq:p1nc}) with $C_{\nu}$  in
(\ref{eq:coneCnc})
is defined by
$$
    l_{\nu}  (x^{\nu}, x^{-\nu}, y^{\nu}): = \theta_{\nu} (x^{\nu},x^{-\nu})
    +\langle y^{\nu}, g_{\nu}(x^{\nu})\rangle
$$
The Nash equilibrium problem with canonically perturbations is expressed as
\begin{equation}\label{GNEP-pnc}
    \textbf{P}_{\nu} (x^{-{\nu}},w,v^{\nu},u^{\nu})\quad \quad\quad \quad
    \begin{array}{ll}
     \min\limits_{x^{\nu} \in X_{\nu}} & f_{\nu}  (x^{\nu}, x^{-{\nu}}; w)-\langle v^{\nu} , x^{\nu}\rangle\\[4pt]
   {\rm s.t.}& F_{\nu} (x^{\nu}; w) - u^{\nu}\in K_{\nu},
   \end{array}
   \end{equation}
where $w \in \mathbb R^d$, $v^{\nu} \in \mathbb{R}^{n_{\nu}}$ and $ u^{\nu}\in {\cal Y}_{\nu}$ are perturbation parameters. For simplicity, we denote
$p=(w;v^1,u^1;\ldots; v^{N},u^{N})\in \mathbb R^d \times \mathbb R^{n_1}\times {\cal Y}_1\times \cdots \times \mathbb R^{n_N}\times {\cal Y}_N$.
Let $\bar w \in \mathbb R^d$ be a fixed parameter such that
$$
f_{\nu} (x^{\nu}, x^{-{\nu}};\bar w)=\theta_{\nu}(x^{\nu}, x^{-{\nu}})\quad  \mbox{ and } \quad F_{\nu} (x^{\nu};\bar w)=g_{\nu}(x^{\nu}).
$$
and the Lagrangian function $l_{\nu}$ becomes
$$
    l_{\nu}  (x^{\nu}, x^{-\nu}, y^{\nu}) = f_{\nu} (x^{\nu};\bar w)
    +\langle y^{\nu}, F_{\nu}(x^{\nu};\bar w)\rangle.
$$
We denote the feasible set of the $\nu$-th player optimization problem (\ref{GNEP-pnc}) by $C_{\nu}(w,u^{\nu})$, namely
$$
C_{\nu}(w,u^{\nu})=\{x^{\nu}\in X_{\nu}: F_{\nu} (x^{\nu}; w) - u^{\nu}\in K_{\nu}\}.
$$
The Lagrangian function associated with (\ref{GNEP-pnc}) is defined by
$$
    L_{\nu}  (x^{\nu}, x^{-\nu}, y^{\nu} ;w,v^{\nu},u^{\nu}): = f_{\nu} (x^{\nu},x^{-\nu}; w)  -\langle v^{\nu} , x^{\nu}  \rangle
    +\langle y^{\nu}, F_{\nu}(x^{\nu}; w)-u^{\nu}\rangle
$$
for $\nu=1,...,N$. Under the basic constraint qualification for problem $\textbf{P}_{\nu} (\bar x^{-{\nu}},\bar w,0,0)$  that
$$
\left.
\begin{array}{r}
0\in N_{X_{\nu}}(\bar x^{\nu})+{\rm D}_{x^{\nu}}F_{\nu}(\bar x^{\nu},\bar w)^*y^k=0\\[4pt]
y^k \in N_{K_{\nu}}(F_{\nu}(\bar x^{\nu}\bar w))
\end{array}
\right
\} \Rightarrow y^k=0,
$$
similar to the discussions before (\ref{KKT-originalD}), we  obtain that the KKT conditions for problem $\textbf{P}_{\nu} (\bar x^{-{\nu}},\bar w,0,0)$ are
\begin{equation}\label{KKT-originalDnc}
\begin{array}{l}
0\in \nabla_{x^{\nu}}L_{\nu}(\bar x, y^{\nu};\bar w,0,0)+N_{X_{\nu}}(\bar x^{\nu}),\\[4pt]
0\in -\nabla_{y^{\nu}}L_{\nu}(\bar x, y^{\nu};\bar w,0,0)+N_{K_{\nu}^{\circ}}(y^{\nu}),
\end{array}
\end{equation}
where $K_{\nu}^{\circ}$ is the polar of $K_{\nu}$.

Then  the set of LGNEs  $S_{\rm Lag}: \mathbb{R}^d\times \mathbb R^{n}\times {\cal Y}\rightarrow \mathbb{R}^{n} \times {\cal Y}$,
   of NEP  whose $\nu$-player solves
   $\textbf{P}_{\nu} (x^{-{\nu}},w,v^{\nu},u^{\nu})$ for fixed  $p$ is expressed as
\begin{equation}\label{KKT-mappingnc}
S_{\rm Lag}(p)=\left
\{(x,y)\in \mathbb R^n \times {\cal Y}:\left(
\begin{array}{l}
0\in \nabla_{x^{\nu}}L_{\nu}(x, y^{\nu};w,v^{\nu},u^{\nu})+N_{X_{\nu}}(x^{\nu})\\[4pt]
0\in -\nabla_{y^{\nu}}L_{\nu}(x, y^{\nu};w,v^{\nu},u^{\nu})+N_{K_{\nu}^{\circ}}(y^{\nu})
\end{array} \right),\nu=1,\ldots,N
\right\}.
\end{equation}
Define $u=(u^1;\ldots;u^N)$, $y=(y^1;\ldots;y^N)$, $X=X_1 \times \cdots\times X_N$ and $K=K_1\times \cdots \times K_N$ and
$$
\textbf{L}(x,y;p)=\left(
\begin{array}{c}
\nabla_{x^1}L_{1}(x, y^{1};w,v^{1},u^{1})\\[2pt]
\cdots\\
\nabla_{x^{N}}L_{N}(x, y^{N};w,v^{N},u^{N})
\end{array}
\right), \textbf{F}(x;w)=\left(
\begin{array}{c}
F_{1}(x^1,;w)\\[2pt]
\cdots\\
F_{N}(x^N,;w)
\end{array}
\right).
$$
Then $S_{{\rm Lag}}(p)$ can be expressed as the following compact form
\begin{equation}\label{Lag-compactc}
S_{\rm Lag}(p)=\left
\{(x,y)\in \mathbb R^n \times {\cal Y}:0\in \left(
\begin{array}{c}
\textbf{L}(x,y;p)\\[4pt]
-\textbf{F}(x;w)+u
\end{array}
\right)+N_{X \times K^{\circ}}(x,y)\right\}.
\end{equation}
Let $\bar p=(\bar w,0,0)$ and $(\bar x,\bar y) \in S_{{\rm Lag}}(\bar p)$. Consider the following set-valued mapping
\begin{equation}\label{Lag-compactclasical}
L_{{\rm Lag}}(\delta)=\left
\{(x,y)\in \mathbb R^n \times {\cal Y}:\delta\in \left(
\begin{array}{c}
\bar q_1 +\textbf{G}x+{\cal A}^*y\\[4pt]
\bar q_2-{\cal A}x
\end{array}
\right)+N_{X \times K^{\circ}}(x,y)\right\}.
\end{equation}
where
\begin{equation}\label{eq:qsclassical}
\begin{array}{l}
\bar q_1=\textbf{l}(\bar x,\bar y)-\textbf{G}\bar x-{\cal A}^*\bar y,\\[4pt]
\bar q_2=-\textbf{g}(\bar x)+{\cal A}\bar x
\end{array}
\end{equation}
and
\begin{equation}\label{eq:dersclassical}
\begin{array}{l}
\textbf{G}=\left[
\begin{array}{cccc}
\nabla^2_{x^1x^1}l_1(\bar x,\bar y^1) & \nabla^2_{x^1x^2}l_1(\bar x,\bar y^1)& \cdots & \nabla^2_{x^1x^N}l_1(\bar x,\bar y^1)\\[4pt]
\nabla^2_{x^2x^1}l_2(\bar x,\bar y^2) & \nabla^2_{x^2x^2}l_2(\bar x,\bar y^2)& \cdots & \nabla^2_{x^2x^N}l_2(\bar x,\bar y^2)\\[4pt]
\vdots &\vdots & \vdots & \vdots\\[4pt]
\nabla^2_{x^Nx^1}l_N(\bar x,\bar y^N) & \nabla^2_{x^Nx^2}l_2(\bar x,\bar y^N)& \cdots & \nabla^2_{x^Nx^N}l_N(\bar x,\bar y^N)
\end{array}
\right],\\[28pt]
{\cal A}=\left(
\begin{array}{cccc}
{\rm D}_{x^1}g_1(\bar x^1) & 0 & \cdots &0\\[4pt]
0 & {\rm D}_{x^2}g_2(\bar x^2) & \cdots &0\\[4pt]
\vdots & \vdots & \vdots & \vdots \\[4pt]
0 & 0) & \cdots &{\rm D}_{x^N}g_N(\bar x^N) \\[4pt]
\end{array}
\right).
\end{array}
\end{equation}
Like Theorem \ref{ThAubin}, we can prove the following result.
\begin{theorem}\label{ThAubinclassical}
  Suppose $(\bar x,\bar y)\in S_{{\rm Lag}}(\bar p)$. Let $f_{\nu}: \mathbb R^n \times \mathbb R^d\rightarrow \mathbb R$ and $F_{\nu}: \mathbb R^{n_{\nu}} \times \mathbb R^d\rightarrow {\cal Y}_{\nu}$ be  twice differentiable with respect to $x$
  in a neighborhood of $(\bar x,\bar w)$. Suppose that $f_{\nu}, F_{\nu}$,  $\nabla_x f_{\nu}$ and  ${\rm D}_xF_{\nu}$  are strictly differentiable with respect to $(x,w)$. Then
  $S_{{\rm Lag}}$ has Aubin property at $\bar p=(\bar w,0,0)$ for $(\bar x,\bar y)$ if and only if
  \begin{equation}\label{GNEP-Aubinclassical}
  \left.
  \begin{array}{rl}
  \textbf{G}^*w_1-\left(
  \begin{array}{c}
  {\rm D}_{x^1}g_1(\bar x^1)^*[w_2]_1\\
  \vdots\\
  {\rm D}_{x^N}g_N(\bar x^N)^*[w_2]_N
  \end{array}
  \right)
   & \in \left(
   \begin{array}{c}
   D^*N_{X_1}(\bar x^1\,|-\nabla_{x^1}l_1(\bar x,\bar y^1)(-[w_1]_1)\\
   \vdots\\
    D^*N_{X_N}(\bar x^N\,|-\nabla_{x^N}l_N(\bar x,\bar y^N)(-[w_1]_N)
    \end{array}
    \right)
   \\[4pt]
 \left(
 \begin{array}{c}
 {\rm D}_{x^1}g_1(\bar x^1)[w_1]_1\\
 \vdots\\
   {\rm D}_{x^N}g_N(\bar x^N) [w_1]_N
   \end{array}
   \right) & \in
   \left(
   \begin{array}{c}
   D^*N_{K_1^{\circ}}(\bar y^1\,|g_1(\bar x^1))(-[w_2]_1)\\
   \vdots\\
   D^*N_{K_N^{\circ}}(\bar y^N\,|g_N(\bar x^N))(-[w_2]_N)
    \end{array}
    \right)
    \end{array}
  \right \} \Longrightarrow w_1=0,w_2=0.
  \end{equation}
  \end{theorem}

\begin{corollary}\label{corSRegularityclassical}

Suppose conditions in Theorem \ref{ThAubin} are satisfied and $\Theta=X \times K^{\circ}$ is polyhedral, then $S_{{\rm Lag}}$ has a Lipschitz continuous single-valued localization at $\bar p=(\bar w,0,0)$ for $(\bar x,\bar y)$ if and only if the implication (\ref{GNEP-Aubinclassical}) holds.
\end{corollary}
Like Theorem \ref{ThIC}, we can prove the following result.
\begin{theorem}\label{ThICclassical}
  Suppose $(\bar x,\bar y)\in S_{{\rm Lag}}(\bar p)$. Let $f_{\nu}: \mathbb R^n \times \mathbb R^d\rightarrow \mathbb R$ and $F_{\nu}: \mathbb R^n \times \mathbb R^d\rightarrow {\cal Y}_{\nu}$ be  twice differentiable with respect to $x$
  in a neighborhood of $(\bar x,\bar w)$. Suppose that $f_{\nu}, F_{\nu}$,  $\nabla_x f_{\nu}$ and  ${\rm D}_xF_{\nu}$  are strictly differentiable with respect to $(x,w)$. Then
  $S_{{\rm Lag}}$ has the isolated calmness property at $\bar p=(\bar w,0,0)$ for $(\bar x,\bar y)$ if and only if
   \begin{equation}\label{GNEP-ICclassical}
  \left.
  \begin{array}{rl}
 - \textbf{G}d_x-\left(
  \begin{array}{c}
  {\rm D}_{x^1}g_1(\bar x^1)^*[d_y]_1\\
  \vdots\\
 {\rm D}_{x^N}g_N(\bar x^N)^*[d_y]_N
  \end{array}
  \right)
   & \in \left(
   \begin{array}{c}
   DN_{X_1}(\bar x^1\,|-\nabla_{x^1}l_1(\bar x,\bar y^1)([d_x]_1)\\
   \vdots\\
    DN_{X_N}(\bar x^N\,|-\nabla_{x^N}l_N(\bar x,\bar y^N)([d_x]_N)
    \end{array}
    \right)
   \\[8pt]
 \left(
 \begin{array}{c}
 {\rm D}_{x^1}g_1(\bar x^1)[d_x]_1\\
 \vdots\\
   {\rm D}_{x^N}g_N(\bar x^N) [d_x]_N
   \end{array}
   \right) & \in
   \left(
   \begin{array}{c}
   DN_{K_1^{\circ}}(\bar y^1\,|g_1(\bar x^1))([d_y]_1)\\
   \vdots\\
   DN_{K_N^{\circ}}(\bar y^N\,|g_N(\bar x^N))([d_y]_N)
    \end{array}
    \right)
    \end{array}
  \right \} \Longrightarrow d_x=0,d_y=0.
  \end{equation}

  \end{theorem}

 In the remaining part of this section, we consider the case when $C_{\nu}$ is defined by both equality and inequality constraints:
    \begin{equation}\label{eq:ineqcnn}
C_{\nu}=\{x^{\nu}\in \mathbb R^{n_{\nu}}:g_{\nu}(x^{\nu})\in K_{\nu}\}\mbox{ with } K_{\nu}=\{0_{q_{\nu}}\} \times \mathbb R^{m_{\nu}-q_{\nu}}_-,
\end{equation}
  where $g_{\nu}: \mathbb R^{n_{\nu}} \rightarrow \mathbb R^{m_{\nu}}$, $\nu=1,\ldots,N$. We define $m=m_1+\ldots+m_N$.

  In this case $S_{{\rm Lag}}(p)$ is be expressed as the following form
\begin{equation}\label{Lag-eqIeqCnn}
S_{\rm Lag}(p)=\left
\{(x,y)\in \mathbb R^n \times \mathbb R^q:
\begin{array}{l}
0=\textbf{L}(x,y;p)\\[4pt]
0= -F_{1,a}(x^1;w)+u^{1,a}\\[4pt]
0\in -F_{1,b}(x^1;w)+u^{1,b}+N_{\mathbb R^{m_1-q_1}_+}(y^{1,b})\\[4pt]
\quad \quad \quad \vdots\\[4pt]
0=-F_{N,a}(x^N;w)+u^{N,a}\\[4pt]
0\in -F_{N,b}(x^N;w)+u^{N,b}+N_{\mathbb R^{m_N-q_N}_+}(y^{N,b})\\[4pt]
\end{array}
\right\},
\end{equation}
where
$$
\textbf{L}(x,y;p)=\left(
\begin{array}{c}
\nabla_{x^1}f_{1}(x;w)+{\cal J}_{x^1}F_1(x^1;w)^Ty^{1}-v^{1}\\[2pt]
\cdots\\
\nabla_{x^N}f_{N}(x;w)+{\cal J}_{x^N}F_N(x^N;w)^Ty^{N}-v^{N}
\end{array}
\right),\,\, F_{\nu}=\left(
\begin{array}{c}
F_{\nu,a}\\[3pt]
F_{\nu,b}
\end{array}
\right), \,\,  u^{\nu}=\left(
\begin{array}{c}
u^{\nu,a}\\[3pt]
u^{\nu,b}
\end{array}
\right),\,\, y^{\nu}=\left(
\begin{array}{c}
y^{\nu,a}\\[3pt]
y^{\nu,b}
\end{array}
\right)
$$
 with $F_{\nu,a}:\mathbb R^{n_{\nu}} \times \mathbb R^d \rightarrow \mathbb R^{q_{\nu}}$, $F_{\nu,b}:\mathbb R^{n_{\nu}} \times \mathbb R^d \rightarrow \mathbb R^{m_{\nu}-q_{\nu}}$, and $u^{\nu,a},y^{\nu,a}\in R^{q_{\nu}}$ and $u^{\nu,b}, y^{\nu,b}\in R^{m_{\nu}-q_{\nu}}$.

For  $(\bar x,\bar y)\in S_{{\rm Lag}}(\bar p)$, define three sets of indices:
 \begin{equation}\label{3indc}
 \begin{array}{l}
 \alpha_{\nu}=\{1,\ldots, q_{\nu}\} \cup\{i: [g_{\nu}]_i(\bar x^{\nu})=0< \bar y^{\nu}_i, i=q_{\nu}+1,\ldots,m_{\nu}\},\\[4pt]
 \beta_{\nu}=\{i: [g_{\nu}]_i(\bar x^{\nu})=0=\bar y^{\nu}_i, i=q_{\nu}+1,\ldots,m_{\nu}\},\\[4pt]
 \gamma_{\nu}=\{i: [g_{\nu}]_i(\bar x^{\nu})<0=\bar y^{\nu}_i, i=q_{\nu}+1,\ldots,m_{\nu}\}.
\end{array}
 \end{equation}
If for any $d \in \mathbb R^n$  satisfying
$$
{\cal J}_{x^{\nu}}g_{\nu,\alpha_{\nu}}(\bar x^{\nu})d^{\nu}=0, {\cal J}_{x^{\nu}}g_{\nu,\beta_{\nu}}(\bar x^{\nu})d^{\nu}\leq0,
d^{\nu}\ne 0, \nu=1,\ldots, N,
$$
one has that
$$
\langle d^{\nu}, \nabla^2_{x^{\nu}x^{\nu}}l_{\nu}(\bar x,\bar y^{\nu})d^{\nu}\rangle >0, \nu=1,\ldots, N,
$$
then $\bar x$ be a local Nash equilibrium point for the GNEP problem.

Using Theorem \ref{ThAubinclassical} and \cite[Theroem 1,Proposition 2]{DRockafellar96}, we can easily obtain the following result.
 \begin{theorem}\label{svalued-lip-nlpclassical}
  Suppose $(\bar x,\bar y)\in S_{{\rm Lag}}(\bar p)$. Let $f_{\nu}: \mathbb R^n \times \mathbb R^d\rightarrow \mathbb R$ and $F_{\nu}: \mathbb R^n \times \mathbb R^d\rightarrow \mathbb R^{m_{\nu}}$ be  twice differentiable with respect to $x$
  in a neighborhood of $(\bar x,\bar w)$. Suppose that $f_{\nu}$  and  $\nabla_x f_{\nu}$ are strictly differentiable with respect to $(x,w)$ and
  $F_{\nu}$ and  ${\cal J}_{x^{\nu}}F_{\nu}$  are strictly differentiable with respect to $(x^{\nu},w)$ for $\nu=1,\ldots,N$. Then $S_{{\rm Lag}}$ has a Lipschitz continuous single-valued localization at $\bar p=(\bar w,0,0)$ for $(\bar x,\bar y)$ if and only if, for any $\nu=1,\ldots, N$, for any partition
  $\beta_{\nu}=\beta_{\nu}^+\cup \beta_{\nu}^0\cup \beta_{\nu}^-$ and $\alpha_{\nu}^+=\alpha_{\nu}\cup \beta_{\nu}^+$,
                   \begin{equation}\label{ec-matrixc}
                   \left.
                 \begin{array}{r}
                 \textbf{G}^Tw_1=\left(
  \begin{array}{r}
  {\cal J}g_{1,\alpha_{1}^+}(\bar x^1)^T[w_2]_{\alpha_{1}^+}+{\cal J}g_{1,\beta^0_{1}}(\bar x^1)^T[w_2]_{\beta_{1}^0}\\
  \vdots\\
  {\cal J}g_{N,\alpha_{N}^+}(\bar x^N)^T[w_2]_{\alpha_{N}^+}+{\cal J}g_{N,\beta^0_{N}}(\bar x^1)^T[w_2]_{\beta_{N}^0}
  \end{array}
  \right)\\[6pt]
                 {\cal J}_{x^{\nu}}g_{\nu,\alpha_{\nu}^+}(\bar x^{\nu})[w_1]_{\nu}=0,\, {\cal J}_{x^{\nu}}g_{\nu,\beta_{\nu}^0}(\bar x^{\nu})[w_1]_{\nu}\leq 0,
                 [w_2]_{\beta_{\nu}^0}\geq0\\[6pt]
                 \nu=1,\ldots N
                 \end{array}
                 \right
                 \} \Longrightarrow\begin{array}{r}
                  w_1=0, (w_2)_{\alpha_{\nu}^+\cup\beta_{\nu}^0}=0,\\
                  \nu=1,\ldots,N,
                  \end{array}
                 \end{equation}
                  where $\textbf{G}$ is defined by (\ref{eq:dersclassical}).
  \end{theorem}

  For every $\nu=1,\ldots, N$, when the strict complementarity condition holds for $\nu$-th player's problem at $(\bar x^{\nu},\bar y^{\nu})$, namely $\beta_{\nu}=\emptyset$, we have the following corollary.
  \begin{corollary}\label{cor:scclassical}
 Suppose $(\bar x,\bar y)\in S_{{\rm Lag}}(\bar p)$. Let $f_{\nu}: \mathbb R^n \times \mathbb R^d\rightarrow \mathbb R$ and $F_{\nu}: \mathbb R^{n_{\nu}} \times \mathbb R^d\rightarrow \mathbb R^{m_{\nu}}$ be  twice differentiable with respect to $x$
  in a neighborhood of $(\bar x,\bar w)$,
  $F_{\nu}$ and  ${\cal J}_{x^{\nu}}F_{\nu}$  are strictly differentiable with respect to $(x^{\nu},w)$ for $\nu=1,\ldots,N$. Assume that $\beta_{\nu}=\emptyset$ for every $\nu=1,\ldots, N$. Then $S_{{\rm Lag}}$ has a Lipschitz continuous single-valued localization at $\bar p=(\bar w,0,0)$ for $(\bar x,\bar y)$ if and only if
  the matrix
                 \begin{equation}\label{ec-matrixScclassical}
                 \left[
                 \begin{array}{ll}
                 \textbf{G}^T & -{\cal A}_1^T\\[4pt]
                 {\cal B}_1^T & 0
                 \end{array}
                 \right
                 ]
                 \end{equation}
                 is nonsingular, where $\textbf{G}$ is defined by (\ref{eq:dersclassical}), ${\cal A}_1$ is defined by
                                  \begin{equation}\label{eq:notationsNccsclassical}
\begin{array}{l}
{\cal A}_1=\left(
\begin{array}{cccc}
{\cal J}_{x^1}g_{1,\alpha_1}(\bar x^1) & 0 & \cdots & 0\\[4pt]
0 &{\cal J}_{x^2}g_{2,\alpha_2}(\bar x^2)& \cdots & 0 \\[4pt]
\vdots & \vdots & \vdots & \vdots \\[4pt]
0& 0 & \cdots & {\cal J}_{x^N}g_{N,\alpha_N}(\bar x^N) \\[4pt]
\end{array}
\right).
\end{array}
\end{equation}
  \end{corollary}
  \begin{remark}\label{rem:LICQclassical}
  In view of Corollary \ref{cor:scclassical}, we have that, if $\beta_{\nu}=\emptyset$ for every $\nu=1,\ldots, N$, then the Lipschitz continuous single-valued localization of $S_{{\rm Lag}}$ at $\bar p=(\bar w,0,0)$ for $(\bar x,\bar y)$ implies that the set of vectors
  $$
   \left\{\nabla g_{\nu,i}(\bar x^{\nu}):i \in \alpha_{\nu}\right\}
  $$
  are linearly independent for every $\nu=1,\ldots,N$.
  \end{remark}
 Now, we propose a set of second-order sufficient optimality conditions for a  vector $\bar x$ being a  Nash equilibrium point.
Suppose $S^{{\rm Lag}}(\bar p)\ne \emptyset$, let  $(\bar x,\bar y)\in S^{{\rm Lag}}(\bar p)$, define
  $$
  {\cal C}(\bar x,\bar y)=\{d\in \mathbb R^n:  {\cal J_{\nu}}g_{\alpha_{\nu}}(\bar x^{\nu})d^{\nu}=0, {\cal J_{\nu}}g_{\beta_{\nu}}(\bar x^{\nu})d^{\nu}
  \leq 0,\nu=1,\ldots,N\}.
  $$
  We propose the following second-order optimality condition:
  \begin{equation}\label{2ndcondclassical}
  d \in {\cal C}(\bar x,\bar y), d\ne 0 \Longrightarrow \langle d, \textbf{G}^Td \rangle>0.
  \end{equation}
  \begin{proposition}\label{a-suffclassical}
 Suppose $(\bar x,\bar y)\in S_{{\rm Lag}}(\bar p)$. Let $f_{\nu}: \mathbb R^n \times \mathbb R^d\rightarrow \mathbb R$ and $F_{\nu}: \mathbb R^{n_{\nu}} \times \mathbb R^d\rightarrow \mathbb R^{m_{\nu}}$ be  twice differentiable with respect to $x$
  in a neighborhood of $(\bar x,\bar w)$. Suppose that $f_{\nu}$  and  $\nabla_x f_{\nu}$ are strictly differentiable with respect to $(x,w)$ and
  $F_{\nu}$ and  ${\cal J}_{x^{\nu}}F_{\nu}$  are strictly differentiable with respect to $(x^{\nu},w)$ for $\nu=1,\ldots,N$. Suppose $S_{{\rm Lag}}(\bar p)\ne \emptyset$ and let  $(\bar x,\bar y)\in S_{{\rm Lag}}(\bar p)$. If the second-order optimality condition (\ref{2ndcondclassical}) holds, then $\bar x$ is a local Nash equilibrium point.
  \end{proposition}
    {\bf Proof}. Similar to the proof of Proposition \ref{a-suff}. \hfill $\Box$\\
  Using the second-order optimality condition (\ref{2ndcondclassical}), we may
   derive the Lipschitz continuous single-valued localization of $S_{{\rm Lag}}$ at $\bar p=(\bar w,0,0)$ for $(\bar x,\bar y)$.

  \begin{proposition}\label{a-suff-icclassical}
   Suppose $(\bar x,\bar y)\in S_{{\rm Lag}}(\bar p)$. Let $f_{\nu}: \mathbb R^n \times \mathbb R^d\rightarrow \mathbb R$ and $F_{\nu}: \mathbb R^{n_{\nu}} \times \mathbb R^d\rightarrow \mathbb R^{m_{\nu}}$ be  twice differentiable with respect to $x$
  in a neighborhood of $(\bar x,\bar w)$. Suppose that $f_{\nu}$  and  $\nabla_x f_{\nu}$ are strictly differentiable with respect to $(x,w)$ and
  $F_{\nu}$ and  ${\cal J}_{x^{\nu}}F_{\nu}$  are strictly differentiable with respect to $(x^{\nu},w)$ for $\nu=1,\ldots,N$. Suppose $S_{{\rm Lag}}(\bar p)\ne \emptyset$ and let  $(\bar x,\bar y)\in S_{{\rm Lag}}(\bar p)$. If
   \begin{itemize}
   \item[{\rm (i)}]the set of vectors $\{\nabla g_i(\bar x^{\nu}):i \in \alpha_{\nu} \cup \beta_{\nu}\}$ are linearly independent for every
   $\nu=1,\ldots,N$;
   \item[{\rm (ii)}]the second-order optimality condition (\ref{2ndcondclassical}) holds.
   \end{itemize}
   Then $\bar x$ is a local  Nash equilibrium point and $S_{{\rm Lag}}$ has a Lipschitz continuous single-valued localization at $\bar p=(\bar w,0,0)$ for $(\bar x,\bar y)$.
  \end{proposition}
  {\bf Proof}. Similar to the proof of Proposition \ref{a-suff-ic}. \hfill $\Box$

  Using Theorem \ref{ThICclassical}, we can easily to obtain the following result about the isolated calmness property  of $S_{{\rm Lag}}$.
  \begin{theorem}\label{th-ic-nlpclassical}
   Suppose $(\bar x,\bar y)\in S_{{\rm Lag}}(\bar p)$. Let $f_{\nu}: \mathbb R^n \times \mathbb R^d\rightarrow \mathbb R$ and $F_{\nu}: \mathbb R^{n_{\nu}} \times \mathbb R^d\rightarrow \mathbb R^{m_{\nu}}$ be  twice differentiable with respect to $x$
  in a neighborhood of $(\bar x,\bar w)$. Suppose that $f_{\nu}$  and  $\nabla_x f_{\nu}$ are strictly differentiable with respect to $(x,w)$ and
  $F_{\nu}$ and  ${\cal J}_{x^{\nu}}F_{\nu}$  are strictly differentiable with respect to $(x^{\nu},w)$ for $\nu=1,\ldots,N$.
   Then $S_{{\rm Lag}}$ has the isolated calmness property at $\bar p=(\bar w,0,0)$ for $(\bar x,\bar y)$ if and only if
  \begin{equation}\label{ic-nlpthclassical}
                  \left.
                 \begin{array}{l}
                 \textbf{G}d_x+
                   \left(\begin{array}{c}
                   {\cal J}_{x^{1}}g_{1,\alpha_{1}}(\bar x^1)^T[d_y]_{\alpha_{1}}+{\cal J}_{x^{1}}g_{1,\beta_{1}}(\bar x^1)^T[d_y]_{\beta_{1}}\\[6pt]
                   \vdots\\[3pt]
                   {\cal J}_{x^{N}}g_{N,\alpha_{N}}(\bar x^N)^T[d_y]_{\alpha_{N}}+{\cal J}_{x^{N}}g_{N,\beta_{N}}(\bar x^N)^T[d_y]_{\beta_{N}}
                   \end{array}
                   \right)=0\\[6pt]
                 {\cal J}_{x^{\nu}}g_{\nu,\alpha_{\nu}}(\bar x^{\nu})d_x^{\nu}=0,\nu=1,\ldots,N\\[6pt]
                  0\geq {\cal J}_{x^{\nu}}g_{\nu,\beta_{\nu}}(\bar x^{\nu})d_x^{\nu} \perp [d_{y^{\nu}}]_{\beta_{\nu}}\geq 0,\nu=1,\ldots,N
                                 \end{array}
                 \right
                 \} \Longrightarrow d_x=0,[d_y]_{\alpha_{\nu}\cup\beta_{\nu}}=0,
                 \end{equation}
                   where $\textbf{G}$ is defined by (\ref{eq:dersclassical}).
  \end{theorem}

\section{Conclusions}
In this paper, we analyze the stability properties of the solution mapping of LGNEs for the generalized Nash equilibrium problems,  in which each player's  strategy has the feasible set $C_{\nu}(x^{-\nu})$ defined by a conic constraint. First of all, we present characterizations of Aubin property, isolated calmness and the Lipschitz continuous single-valued localization property of this solution mapping, using coderivative criterion and graph derivative criterion. We use the results in the general setting to derive  characterizations of the Lipschitz continuous single-valued localization when each player's feasible set is defined by equality and inequality constraints. Secondly, we use the results for the general setting to analyze stability property for shared constrained generalized Nash equilibrium problems. We characterize
the Aubin property and isolated calmness of the solution mapping of consensus LGNEs. Finally, we apply the general stability results to the classical Nash equilibrium problems with conic constraints for players.

 There are numerous interesting topics for future research. The most critical one is how to leverage the stability properties of LGNEs to analyze the stability properties of the set of Nash equilibriums under canonical perturbation.
Rockafellar \cite{Rock24} adopted a powerful framework to analyze the stability properties of generalized Nash equilibrium problems (GNEPs) and considered three stability notions, namely full stability, tilt stability, and the near tilt or full stability of local equilibria. Additionally, what about the relationships between these three properties and the stability properties of the LGNEs studied herein? How can we derive various stability properties using individual players' second-order optimality conditions that are weaker than those employed in this paper?
Finally, Chen et al. \cite{Chen2025a, Chen2025b} recently established the equivalence between the Aubin property and strong regularity for both nonlinear second-order cone programming and nonlinear semidefinite programming. This equivalence property is worth studying for the stability  of generalized Nash equilibrium problems.

 
\end{document}